\documentclass[10pt]{amsart}
\usepackage{amsmath,amssymb,amsfonts, amscd}
\usepackage{amsbsy}
\usepackage{amsthm}
\usepackage{amstext}
\usepackage{amsopn}
\usepackage{hyperref}
\usepackage{mathrsfs}
\usepackage{graphicx}
\usepackage{latexsym}
\usepackage[T1]{fontenc}
\usepackage{color}

\usepackage[abs]{overpic}

\newcommand{\Hmm}[1]{\leavevmode{\marginpar{\tiny%
$\hbox to 0mm{\hspace*{-0.5mm}$\leftarrow$\hss}%
\vcenter{\vrule depth 0.1mm height 0.1mm width \the\marginparwidth}%
\hbox to 0mm{\hss$\rightarrow$\hspace*{-0.5mm}}$\\\relax\raggedright #1}}}

\newcommand{\Bc}{\mathcal{B}}

\newcommand{\Sc}{\mathcal{S}}
\newcommand{\Uc}{\mathcal{U}}

\newcommand{\ve}{\varepsilon}

\newtheorem{thm}{Theorem}[section]
\newtheorem{cor}[thm]{Corollary}

\newtheorem{lemma}[thm]{Lemma}

\theoremstyle{definition}

\newtheorem{eg}[thm]{Example}
\newtheorem*{rem}{Remark}

\newcommand{\R}{{\mathbb R}}
\newcommand{\C}{{\mathbb C}}

\newcommand{\N}{{\mathbb N}}

\newcommand{\al}{{\alpha}}

\newcommand{\eps}{{\varepsilon}}
\newcommand{\gm}{{\gamma}}

\newcommand{\Om}{{\Omega}}

\newcommand{\ka}{{\kappa}}
\newcommand{\si}{{\sigma}}
\newcommand{\lm}{{\lambda}}
\newcommand{\ph}{{\varphi}}

\begin{document}
\title[Eigenvalue asymptotics and unique continuation of eigenfunctions]{Eigenvalue asymptotics  and unique continuation of eigenfunctions on planar graphs}

\author[M. Bonnefont]{Michel Bonnefont} \address{Univ. Bordeaux, CNRS, Bordeaux INP, IMB, UMR 5251,  F-33400, Talence, France} \email{michel.bonnefont@math.u-bordeaux.fr}

\author[S. Gol\'enia]{Sylvain Gol\'enia} \address{Univ. Bordeaux, CNRS, Bordeaux INP, IMB, UMR 5251,  F-33400, Talence, France} \email{sylvain.golenia@math.u-bordeaux.fr}

\author[M. Keller]{Matthias Keller}
\address{Matthias Keller,  Universit\"at Potsdam, Institut f\"ur Mathematik, 14476  Potsdam, Germany}
\email{matthias.keller@uni-potsdam.de}

\begin{abstract} \noindent
We study planar graphs with large negative curvature outside of a finite set and the spectral theory of  Schr\"odinger operators on these graphs. We obtain  estimates  on the first and second order term of the eigenvalue asymptotics. Moreover, we prove a  unique continuation result for eigenfunctions and decay properties of general eigenfunctions. The proofs rely on a detailed {analysis} of the geometry which employs a Copy-and-Paste procedure based on the Gau\ss-Bonnet theorem.
\end{abstract}

\date{\today}
%\maketitle
\maketitle
%\tableofcontents

% \frenchspacing
%%%%%%%%%%%%%%%%%%%%%%%%%%%%%%%%%%%%%%%%%%%%%%%%%%%%%%%%%%%%
% ABSTRACT
%%%%%%%%%%%%%%%%%%%%%%%%%%%%%%%%%%%%%%%%%%%%%%%%%%%%%%%%%%%%

\section{Introduction}

In recent years  consequences  of  curvature bounds on the geometry and spectral theory of graphs have been intensively studied. For planar graphs a notion of curvature was introduced by Stone \cite{S1} going back to ideas  to Alexandrov and even Descartes. Recently, the study of this curvature  gained some momentum. For positive and non-negative curvature geometric consequences and harmonic functions were studied in \cite{CC,DM,HJ,HJL,S1,Z}. On the other hand, the geometry of negative and non-positive curvature was investigated in \cite{BP1,BP2,H,K2,KP,O,W} as well as for spectral consequences see \cite{BHK,KLPS,K1,K2}. For more recent work on sectional curvature of polygonal complexes see \cite{KPP}.

The  subject of this paper are  planar graphs  with large negative curvature  outside of a finite set and we are interested in the spectral theory of the Laplacian or more general that of Schr\"odinger operators. Especially, we study the asymptotics of eigenvalues, existence of eigenfunctions of compact support and  decay properties of eigenfunctions in general.

Let us discuss the results of the paper in the light of the existing literature. 
In \cite{K1} it was proven that if the curvature tends to negative infinity uniformly then the spectrum of the Laplacian is purely discrete.  The first order term of the eigenvalue asymptotics  was obtained \cite{BGK,G} for so called  sparse  graphs which include planar graphs. Here we get a hold on the second order term of the asymptotics of the eigenvalues in the case of planar graphs with uniformly decreasing curvature, see Theorem~\ref{c:spareinequality1} and Corollary~\ref{c:spareinequality2}.

Next, we turn to eigenfunctions. In \cite{KLPS,K2} unique continuation results for graphs with non-positive corner curvature were shown. However, these results are rather delicate and fail to hold for example for the Kagome lattice which has non-positive vertex curvature only, see \cite{KLPS}. Moreover, we also present an example that {the} failure of the curvature assumption on a finite set can lead to infinitely many compactly supported eigenfunctions, see Section~\ref{s:counter-example}. On the other hand, we show that if the curvature is sufficiently negative  outside of a finite set, then compactly supported eigenfunctions can occur in a finite region only, see Theorem~\ref{t:main2}.

Finally, we prove Agmon estimates as they were recently obtained in \cite{KePo} to give decay results on general eigenfunctions, Theorem~\ref{t:eigenfunctions}.

To prove these results we carefully study the geometry of graphs with large degree outside of finite set.  The underlying philosophy (which is made precise later in the paper) is that  we can continue such a planar graph to a tessellation with non-positive corner curvature after generously removing the set of positive curvature. 

While the geometric results are mainly phrased without mentioning curvature the proofs  make use of the Gau\ss-Bonnet theorem -- which essentially involves curvature.  
Firstly, the geometric results include statements about the  sphere structure of the graph sufficiently far outside. These considerations yield immediately a Cartan-Hadarmard type result about  continuation of geodesics, see Theorem~\ref{t:main} for these results. While we also recover some of the results of \cite{BP1,BP2} our approach is independent of theirs.  
Secondly we investigate the existence of spanning trees that are in some sense close to the original graphs. In particular, we show that there exist spanning trees that are bounded perturbations of the original graph, Theorem~\ref{t:spanning}. 

These geometric results are then applied to the study of the spectral theory of Schr\"odinger operators in the subsequent.

The paper is structured as follows. In the next subsection we introduce the basic notions and in the subsequent two subsections we present the geometric and spectral results.
The proof of the geometric result relies heavily on a so called Copy-and-Paste procedure presented in Section~\ref{s:CopyPaste}. Next, we closely study the case of triangulations in Section~\ref{s:triangulation}. Then the geometric results follow rather directly from considering triangulation supergraphs. A result about spanning trees is proven in Subsection~\ref{s:proofspanning} and the result about continuing a graph with negative curvature outside a finite set to a non-positively curved tessellation is shown in Section~\ref{s:general-planar}. The unique continuation result is also proven in this section.  
Finally, in Section~\ref{s:sparseinequality} discrete spectrum, the asymptotics of eigenvalues and the decay of eigenfunctions are proven.

\subsection{Set up and definitions}\label{s:setup}
Let an infinite connected simple graph  $G=(V,E)$ be given. The \emph{degree} $deg(v) $ of a vertex $v \in V$ is the number of adjacent vertices. We  assume $2\le\mathrm{deg}(v)<\infty$ for all $v\in V$. We call a sequence of vertices $(v_{0},\ldots, v_{n})$ a \emph{walk} of length $n$ if $v_{0}\sim\ldots\sim v_{n}$, where $v\sim w$ denotes that $v$ and $w$ are adjacent. 

We denote by $d$ the \emph{natural graph distance} on $G$ which is the length of the shortest walk between two vertices.

We fix a vertex $o\in V$ which we call the \emph{root}. For $r\ge0$, we define the \emph{sphere} with respect the natural graph distance by
\begin{align*}
S_{r}:=S_{r}(o):=\{v\in V\mid d(o,v)=r\}.
\end{align*}
The distance \emph{balls} are  defined as
\begin{align*}
B_{r}:=B_{r}(o):=\{v\in V\mid d(o,v)\leq r\}.
\end{align*}
For a vertex $v\in S_{r}$, $r\ge0$, we call $w\in S_{r\pm1}$, $v\sim w$ a \emph{forward/backward neighbor} and denote
\begin{align*}
\deg_{\pm}(v)&:=\{w\in S_{r\pm1}\mid w\sim v\}\quad\mbox{and}\quad
\deg_{0}(v):=\{w\in S_{r}\mid w\sim v\}.
\end{align*}

We assume that   $G$ is a planar graph which is embedded into an orientable topological surface $\mathcal{S}$ homeomorphic to $\R^{2}$. We assume that  the embedding of $G$ is \emph{locally finite}, that is for every point in $\mathcal{S}$ there is a neighborhood which intersects only finitely many edges.

From now on, when we speak about \emph{planar graphs} we always assume to have an  infinite connected simple planar graph which admits a locally finite embedding.

We associate to $G$ the set of \emph{faces} $F$ whose elements are defined as the closures of the connected components of $\mathcal{S}\setminus \bigcup E$, i.e., the connected components of $ \mathcal{S} $ after removing the edge segments.  The \emph{boundary} of a face $f$ is defined as the elements of $V$ whose image belongs to $f$. A \emph{boundary walk} of $ f $ is a closed walk which  visits every vertex of $ f $. The length of the shortest  boundary walk is called the \emph{degree} $\mathrm{deg}(f)$ of the face $f\in F$ and if no closed boundary walks exists we say that $ f $ has infinite degree. In what follows we do not distinguish between the graph and its embedding.

The set of \emph{corners} $C(G)$ is given as the set of pairs $(v,f)\in V\times F$ such that $v$ is contained in $f$. The \emph{degree} $|(v,f)|$ of a corner $(v,f)$ is  the minimal number of times the vertex $ v$ is
met by a boundary walk of $f$.
The \emph{corner curvature} $\ka_{C}:C(G)\to \R$ is given by
\begin{align*}
    \ka_{C}(v,f) :=\frac{1}{\mathrm{deg}(v)}-\frac{1}{2} +\frac{1}{\mathrm{deg}(f)}.
\end{align*}
This quantity was first introduced in \cite{BP1,BP2} for tessellations  and in \cite{K2} for general planar graphs.
Summing over all corners of a  vertex gives the \emph{vertex curvature} $\ka:V\to\R$
\begin{align*}
    \ka(v):=\sum_{(v,f)\in C(G)}|(v,f)|\ka_{C}(v,f)
\end{align*}
This quantity was first defined in \cite{S1} for tessellations following ideas of Alexandrov and for general planar graphs in \cite{K2}. In \cite{K2} a Gauß-Bonnet formula for this curvature is shown. Moreover, one has since $ \deg(v)=\sum_{f\ni v}|(v,f)| $, for all $ v\in V $
\begin{align*}
	\ka(v)=1-\frac{\mathrm{deg}(v)}{2} +\sum_{f\in F, f\ni v}|(v,f)|\frac{1}{\mathrm{deg}(f)}.
\end{align*}
The most interesting examples are tessellations which are discussed in a slightly more general form in Section~\ref{s:counter-example}.

We continue by introducing some more notation needed for the paper. For two walks $p=(v_{0},\ldots,v_{n})$ and $q=(w_{0},\ldots,w_{m})$ with $v_{n}=w_{0}$ or $v_{0}=w_{m}$, we denote by $p+ q$ the walk $(v_{0},\ldots,v_{n},w_{1},\ldots,w_{m})$ if $v_{n}=w_{0}$ or $(v_{1},\ldots,v_{n},w_{m-1},\ldots,w_{0})$ if $v_{n}=w_{m}$.
A walk $(v_{0},\ldots,v_{n})$  is called a \emph{path} if the vertices in a walk  are pairwise different except for possibly $ v_0=v_{n} $.  We say that $n$ is the {length} of the path. Moreover, a walk $(v_{n})  $ is called a \emph{geodesic} if $ d(v_{0},v_{n})=n $ for all $ n $.  
For a walk $p=(v_{0},\ldots,v_{n})$, we denote its vertex set by $V(p)=\{v_{0},\ldots,v_{n}\}$ and call it the \emph{trace} of $ p $. We call the vertices $v_{1},\ldots,v_{n-1}$ the \emph{inner vertices} and $v_{0},v_{n}$ the \emph{outer vertices} of the walk $p$. We call $p$ \emph{closed} if $v_0= v_{n}$. Note that the definition of a path does not allow repetition of vertices apart from the beginning and the ending vertex. To stress this we sometimes refer to closed paths also as simply closed paths.

Each simply closed path $p$ in the graph induces a simply closed curve with image $ \gamma(p) $ in the surface $\mathcal{S}$ where the graph is embedded.
By Jordan's curve theorem, this induces a partition of $\Sc$  as
$$ \Sc=\Bc (p)\cup \gamma(p) \cup \Uc(p), $$  where $\Bc(p)$  and $\Uc(p)$ are respectively the bounded and the unbounded connected component of $\Sc \setminus \gamma(p)$.

For a subset $W\subseteq V$, let $G_{W}$ be the \emph{induced subgraph} $(W,E_{W})$, where $E_{W}\subseteq E$ is the set of edges with beginning and end vertex in $W$. We say that $G_{W}$ has a \emph{closed boundary path} if there is a closed path $p$  within the graph $G_{W}$ such that $\Bc(p)\cap V=W$. 
Every vertex in $W$ not contained in a boundary path is called an \emph{interior vertex} of $G_{W}$. Indeed, boundary walks are unique up to enumeration.

When we consider two planar graphs $G$ and $G'$ at the same time we denote the degree on $G'$ by $\mathrm{deg}'$ or $\mathrm{deg}^{(G')}$, the curvatures by $\ka_{C}'$ or $\ka_{C}^{(G')}$, $\ka'$ or $\ka^{(G')}$ and the natural graph distance by $d'$ or $d^{(G')}$.

%%%%%%%%%%%%%%%%%%%%%%%%%%%%%%%%%%%%%%%%%%%%%%%%%%%%%%%%%

\subsection{Geometric results}

{In this work we first show that planar graphs with vertex degree large outside a finite set are in some sense really
 close to tree graphs. We shall consider two situations. First, we consider  $ \deg\ge 6 $ for all vertices except possibly for  the root 
 and secondly  $ \deg \ge 7 $ outside of a finite set.}

The first theorem  is a Cartan-Hadarmard type theorem.  This says that (sufficiently long) geodesics can be continued indefinitely which is equivalent to the function $d(o,\cdot)  $ not having local maxima (outside of a finite set). In the literature this is also referred to as absence or emptiness of the cut-locus (which is the set where $ d(o,\cdot) $ attains its local maxima), \cite{BP1,BP2}.

Furthermore the theorem includes a  remarkable structural statement about  distance spheres. To this end, we 
say a subset $ W $ of a planar graph $ G $ can be \emph{cyclically ordered} if there is planar supergraph {$G'$ of $G$} such that $ W $ is the trace of a simply closed path {of $G'$}.

\begin{thm}\label{t:main}Let $G=(V,E)$ be a planar graph, such that one of the following conditions hold:
	\begin{itemize}
		\item [(a)] $ \deg\ge 6 $ outside of the root $o$.
		\item[(b)]  $\deg\ge 7$ outside of some finite set.
	\end{itemize}	
Then, then there exists of a finite set $ K\subseteq V $ (which can be chosen to be $K= \{o\} $ in case (a)) such that  for all $ v\in V\setminus K $
		\begin{align*}
		\deg_{0}(v)\leq 2  \quad\mbox{and}\quad  1\leq    \deg_{-}(v)\leq 2.
		\end{align*}
In particular, any geodesic reaching $ V\setminus K $ from $ o $ can be continued indefinitely.
Furthermore, all the spheres outside of $ V\setminus K $ can be cyclically ordered.
\end{thm}

Observe that parts of the results of (a) are already included in \cite{BP1,BP2} since $ \deg\ge 6 $ implies $ \kappa_{C}\leq 0 $ but our proof  follows a completely different strategy. However, our techniques also allow  us to change our graphs by replacing  a finite set with a vertex such that they become graphs with $ \kappa_{C} $ everywhere. This is discussed in detail in Section~\ref{s:counter-example}.

The second  consequence is that   planar graphs with large vertex degree are  close to some of their spanning trees.

\begin{thm}\label{t:spanning} Let $G$ be a planar graph, such that one of the following conditions hold:
	\begin{itemize}
		\item [(a)] $ \deg\ge 6 $ outside of the root $o$.
		\item[(b)]  $\deg\ge 7$ outside of some finite set.
	\end{itemize}
 Then, there exists a spanning tree $T$ of $G$ such that outside of a finite set  the vertex degrees of $T$ and $G$ differs at most by $4$, where the finite set is empty in case (a). Furthermore, $ T $ and $ G $ have the same sphere structure. 
\end{thm}

\begin{rem} The existence of spanning trees with certain properties is also treated in  \cite{BS,FP}. 
\end{rem}

\begin{rem}
	Our techniques of proof allow us to quantify the finite set in the theorems, Theorem~\ref{t:main} and Theorem~\ref{t:spanning}. In fact, given the radius of the ball out of which the degree is larger than $ 7 $, one can give an estimate on the radius of the ball such that outside of this ball the statements hold.
\end{rem}

\begin{rem} It would be  interesting to study whether the criteria $ \deg\ge6 $ outside of the root or $\deg\ge7$ outside of a finite set can be replaced by a weaker curvature type assumption. In the case of triangulations  $\deg\ge6$ is equivalent to $\ka_{C}\leq0$ and $\deg\ge7$ is equivalent to $\ka_{C}<0$. It remains an open question which of the results still hold for $ \deg\ge 6 $ outside of a finite set.
\end{rem}

The proof of the   geometrical  results above for general graphs  is given in Section~\ref{s:general-planar}. 
It will follow from the case of planar triangulations  by an  embedding into a triangulation supergraph.
The case of triangulation is investigated in Section \ref{s:triangulation}.
It uses a Copy-and-Paste procedure given in Section~\ref{s:CopyPaste} 
and a fine study of an adapted new sphere structure.

%%%%%%%%%%%%%%%%%%%%%%%%%%%%%%%%%%%%%%%%%%%%%%%%%%%%%%%%%

\subsection{Spectral consequences}
{In this section, we  turn to some spectral consequences for the Laplacian on $ \ell^{2}(V) $. We introduce some notation first.}

Denote the space of square summable real valued functions by $\ell^{2}(V)$, the corresponding scalar product by $\langle\cdot,\cdot\rangle$ and the norm by $\|\cdot\|$.

We consider the Laplace operator $\Delta=\Delta_{G}$   defined as
\begin{align*}
D(\Delta):=\{\ph\in\ell^{2}(V)\mid& (v\mapsto \sum_{w\sim v}(\ph(v)-\ph(w)))\in\ell^{2}(V)\}\\
    \Delta\ph(v)&:=\sum_{w\sim v}(\ph(v)-\ph(w)).
\end{align*}	
The operator is positive and selfadjoint (confer \cite[Theorem~1.3.1.]{Woj1}).

For a function $g:V\to\R$, we denote with slight abuse of notation the operator of multiplication  again by $g$  and
\begin{align*}
	g(\ph):=\sum_{v\in V}g(v)\ph(v)^{2}
\end{align*}
and for $\ph\in C_{c}(V)$ (which are the real  valued  functions of compact support).

 For two self-adjoint operators $A$ and $A'$ on $\ell^{2}(V)$ and a subspace $D_{0}\subseteq D(A)\cap D(A')$ we write $A\leq A'$ on  $D_{0}$ if $\langle A\ph,\ph\rangle\leq \langle A'\ph,\ph\rangle$ for $\ph\in D_{0}$.

For a function $q$, we let the positive and negative part be given by $q_{\pm}=\max\{\pm q,0\}$. We denote by $K_{\al}$, $\al\in(0,1]$, the class of potentials $q:V\to\R$ such that there is $C\ge0$  such that
\begin{align*}
     q_{-}\leq \al   (\Delta+q_{+})+C,\qquad \mbox{on $C_{c}(V)$.}
\end{align*}
As the operator  $(\Delta+q)\vert_{C_{c}(V)}$ is symmetric and bounded from below for $q$ in $K_{\al}$, $\al\in(0,1]$, there exists the Friedrich extension which we also denote  by $\Delta+q$.

For a self-adjoint operator $A$ which is bounded from below, we denote the discrete eigenvalues below the bottom of the essential spectrum by $\lm_{n}(A)$ in increasing order counted with multiplicity for all $ n\ge0 $ for which they exist.
We denote
\begin{align*}
	d_{n}=\lambda_{n}(\deg+q),\qquad n\ge0.
\end{align*}
Furthermore, we use the Landau notation $ o(a_{n}) $ for a sequences $(b_{n} ) $ such that $ b_{n}/a_{n}\to 0$ as $ n\to \infty $.

The following theorem is the main result about the asymptotics of eigenvalues.

\begin{thm}\label{c:spareinequality1} Let $G$ be a planar graph and $q\in K_{\al}$, $\al\in(0,1)$. Then the spectrum of $\Delta+q$ is purely discrete if and only if
\begin{align*}
    \sup_{K\subset V\,\mathrm{finite}}\inf_{v\in V\setminus K}(-\ka(v)+q(v))=\infty.
\end{align*}
In this case and if $q\ge0$
\begin{align*}
    d_{n}-2\sqrt{d_{n}}-o(\sqrt{d_{n}}) \leq \lm_{n}(\Delta+q)\le d_{n}+2\sqrt{d_{n}}+o(\sqrt{d_{n}}).
\end{align*}
% which is to be interpreted as
%\begin{align*}    \lim_{n\to\infty} \frac{\lm_{n}(\Delta+q)}{\lm_{n}(\deg+q)}=1\end{align*}
%and
%\begin{align*}-2\leq \liminf_{n\to\infty}  \frac{\lm_{n}(\Delta+q)-\lm_{n}(\deg+q)}{\sqrt{\lm_{n}(\deg+q)}}\leq \limsup_{n\to\infty}  \frac{\lm_{n}(\Delta+q)-\lm_{n}(\deg+q)}{\sqrt{\lm_{n}(\deg+q)}}\leq 2.\end{align*}
\end{thm}

\begin{rem}(a) The first part of the theorem above was announced in \cite{K3} and is a unification of \cite[Theorem~3]{K1} and \cite[Corollary~21]{KL2} for Schr\"odinger operators on planar graphs.
The second part  improves the considerations of \cite{BGK,G} by giving the second order term on the eigenvalue asymptotics.\\
(b) For potentials $ q $ in $\bigcap_{\al\in(0,1)}K_{\al}$ instead of $ q\ge0 $, one has still the same first term of the eigenvalue asymptotics,  see \cite{BGK}.
\end{rem}

In the case of planar graphs with constant face degree we can even prove bounds with an even more geometric flavor. For $k\ge3$, denote
$$\gm(k):=2\pi\frac{2(k-2)}{k},$$
that is, if a face $f\in F$ with $\deg(f)=k$ is a regular $k$-gon, then $\gm(k)$ is the inner angle of $f$. Moreover,  denote
\begin{align*}
    \ka_{n}=-\lm_{n}(-\ka),\quad n\ge0,	
    \end{align*}
and in case there are infinitely many $ \ka_{n} $ (which are a decreasing sequence) we let
\begin{align*}
    \ka_{\infty}=\lim_{n\to\infty}\ka_{n}.
\end{align*}

\begin{cor}\label{c:spareinequality2}
 Let $G$ be a planar graph and
suppose the face degree is constantly $k$ outside of some finite set.
The operator $\Delta$ has purely discrete spectrum if and only if
\begin{align*}
    \ka_{\infty}=-\infty.
\end{align*}
In this case, $\ka_{n}\le0$ for large $n$ and
\begin{align*}
    -\frac{2\pi}{\gm(k)}\ka_{n}-2\sqrt{-\frac{2\pi}{\gm(k)}\ka_{n}} -o(\sqrt{\kappa_{n}})\leq \lm_{n}(\Delta)\leq  -\frac{2\pi}{\gm(k)}\ka_{n}+2\sqrt{-\frac{2\pi}{\gm(k)}\ka_{n}}+o(\sqrt{\kappa_{n}}).
\end{align*}
%which is to be interpreted as in Theorem~\ref{t:sparseinequality} above. \Hmm{\green{above there is only 1 in front of the $\lambda_n(\deg+q)$, I guess you mean I should replace it with $\frac{2\pi}{\gm(k)}\ka_{n}$. The problem comes from the no-definition of $\lesssim$}. You are right that there is no definition of $ \lesssim $. I believe that there is no elegant and canonical way to give a precise definition. If you feel uncomfortable as it is we may just write what comes after the "which is to be interpreted as" in the theorem above as well as above. An alternative is to use Landau symbols $ o(\sqrt{\kappa_n}) $ instead but at least in the theorem above we will need two lines not fit in one line ...}
\end{cor}

The proofs of the preceding theorems and corollaries are given in Section~\ref{s:sparseinequality}.

From the results above  we learn that in the case of uniformly unbounded curvature the spectrum consists of discrete eigenvalues. 
The following corollary is a unique continuation result telling us that outside of a compact set eigenfunctions have unbounded support.

\begin{thm}\label{t:main2} Let $G$ be a planar graph. Assume
\begin{align*}
   \ka_{\infty}=-\infty.
\end{align*}
Then, outside of a finite set there are no eigenfunctions of compact support of $\Delta+q$ for all $q\in K_{1}$. In particular, there are at most finitely many linearly independent eigenfunctions of compact support.
\end{thm}

\begin{rem}\label{r:main2} (a) In \cite[Theorem~1]{KLPS}, \cite[Theorem~9]{K2} it is proven for tessellations and locally tessellating graphs that $\ka_{C}\leq0$ implies absence of compactly supported eigenfunctions for the operator $\Delta+q$.
We emphasize that Theorem~\ref{t:main2} is not a simple perturbation result of \cite{K2,KLPS}. Indeed, unique continuation is a rather subtle issue on discrete spaces. For example there are spaces  with $\ka\le0$ and which admit compactly supported eigenfunctions see e.g. \cite{BP1}. See also Example~\ref{e:example1} in Section~\ref{s:unique}.

(b) Validity of the theorem does not depend on the particular choice of $\Delta+q$ but holds for arbitrary nearest neighbor operators (i.e.,  with arbitrary complex coefficients for the edges and arbitrary complex potentials), see Theorem~\ref{t:EFcompactsupport-triang} in Section~\ref{s:unique}.

(c)  From the proof of Theorem~\ref{t:main2}, we can deduce an explicit estimate on the size of the set where compactly supported eigenfunctions can be supported, see Theorem~\ref{t:EFcompactsupport_general}.
\end{rem}

For the other eigenfunctions we obtain  a result on the decay which is based on Agmon type estimates as they are developed in  \cite{KePo}. In Section~\ref{s:sparseinequality}, we give a simplified proof which is adapted to the situation of planar graphs.

\begin{thm}\label{t:eigenfunctions} Let $G$ be a planar graph. Assume
	\begin{align*}
		\ka_{\infty}=-\infty.
	\end{align*}
	Then, any eigenfunction $ u\in D(\Delta) $ of $ \Delta $ satisfies
	\begin{align*}
		\al^{ d(o,\cdot)}u\in \ell^{2}(X,\deg),
	\end{align*}
for all $0< \al<{1 +\sqrt{2}}$.
\end{thm}

%%%%%%%%%%%%%%%%%%%%%%%%%%%%%%%%%%%%%%%%%%%%%%%%%%%%%%%%
\section{Copy-and-Paste Lemmas}\label{s:CopyPaste}
In this section we prove a lemma that shows that certain subgraphs imply the presence of positive curvature. The proof works by making several copies of the subgraph and pasting them along the boundary path. The resulting graph can be embedded in the two dimensional sphere. We then apply a discrete Gau\ss-Bonnet theorem. Similar ideas can  be   found in \cite{K2}.

\begin{lemma}[Copy-and-Paste Lemma]\label{l:CopyPaste}
Let $G'=(V',E')$ be a subgraph of a planar graph  $G=(V,E)$ with a simply closed boundary path $p$ {in $G'$} such that there are (at most) three vertices $v_{0},v_{1},v_{2}\in V(p)$ such that
\begin{itemize}
  \item [(a)] $\deg'(v)\ge 4$ for all $v\in V(p)\setminus\{v_{0},v_{1},v_{2}\}$,
  \item [(b)] $\deg'(v_{0})\ge 3$,
  \item [(c)] $\deg'(v_{1}),\deg'(v_{2})\ge 2$.
\end{itemize}
Then, there is  $v\in V'\setminus V(p)$ such that
\begin{align*}
    \ka(v)=\ka'(v)>0.
\end{align*}
\end{lemma}
\begin{proof}
The proof uses by a copy and paste procedure applied to $G'$ with boundary path $ p =p_0+p_1+p_2$ which is illustrated in Figure~\ref{f:CopyPaste}.
 \begin{figure}[!h]
\scalebox{0.6}{\includegraphics{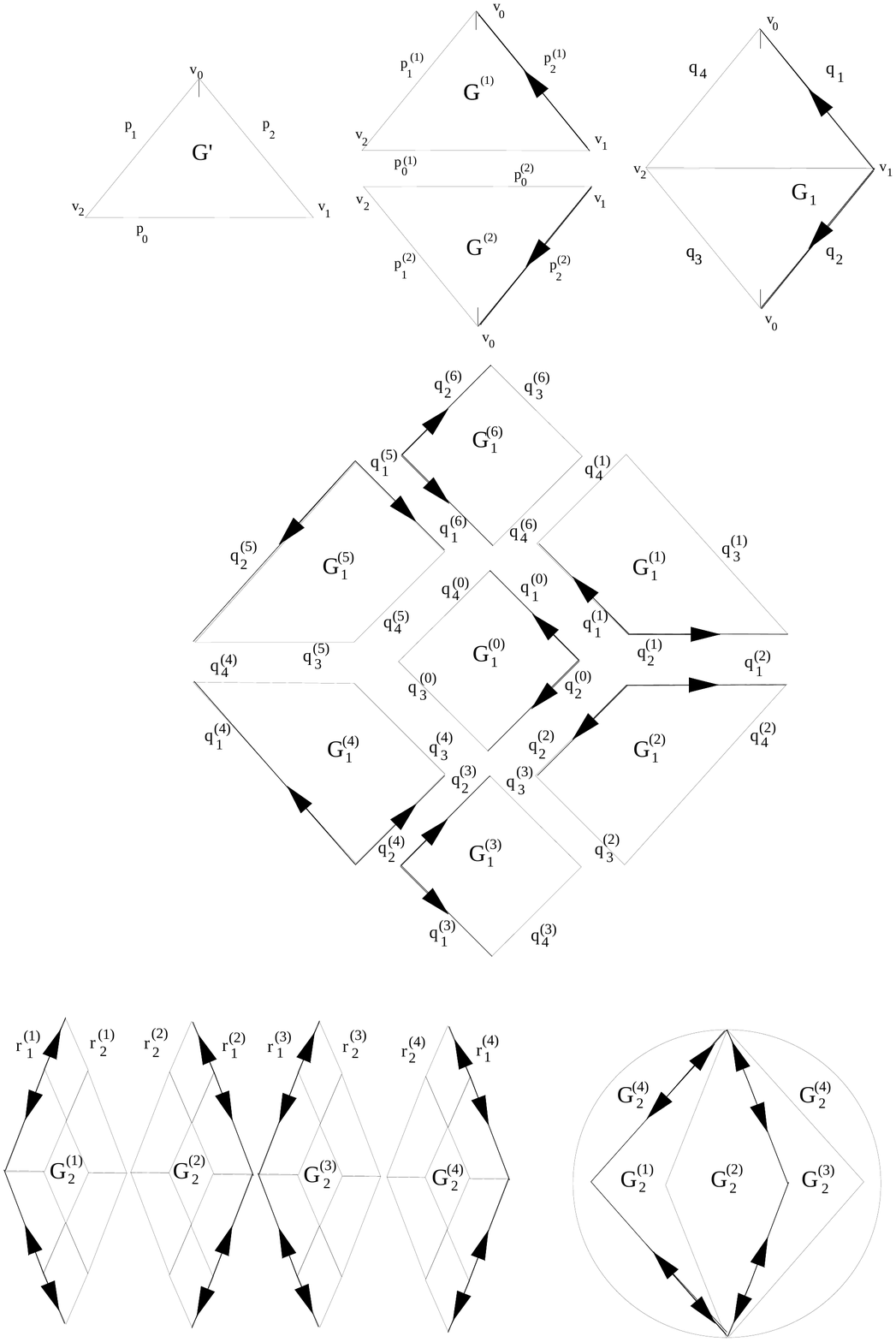}}
\caption{\label{f:CopyPaste} An
illustration of the copy and paste procedure. (The arrows only indicate the orientation of the subpaths.)}
\end{figure}

We denote the subpath of $p$ from $v_{1}$ to $v_{2}$ by $p_{0}$, the subpath of $p$ from $v_{2}$ to $v_{0}$ by $p_{1}$ the one from $v_{0}$ to $v_{1}$ by $p_{2}$. We make two copies $G^{(1)},G^{(2)}$ of $G'$ and denote the corresponding copies of $p_{j}$ by $p_{j}^{(1)}$ and $p_{j}^{(2)}$, $j=0,1,2$. We paste $G^{(1)}$ and $G^{(2)}$ along $p_{0}^{(1)}$ and $p_{0}^{(2)}$ (after reflecting $G^{(2)}$ -- where reflecting always means with respect to the path along the graphs are pasted).
We denote the resulting graph by $G_{1}$. We denote the boundary path of $G_{1}$ in the following way: Denote $p_{1}^{(1)}$ by $q_{4}$, $p_{1}^{(2)}$ by $q_{3}$, $p_{2}^{(2)}$ by $q_2$ and $p_{2}^{(1)}$ by $q_1$. 

%%%

Let us keep track of the vertex degrees in $G_{1}$:
\begin{itemize}
  \item All vertices in the glued path of $p_{0}^{(1)}$ and $p_{0}^{(2)}$ have degree at least $6$ in $G_{1}$.
  \item The copies of the vertices $v_{1}$ and $v_{2}$ in $G^{(1)}$  and $G^{(2)}$ that are merged have now vertex degree at least $3$ in $G_{1}$.
  \item The two copies of $v_{0}$ in $G_{1}$ have still vertex degree at least $3$.
  \item All other vertices in the boundary path have degree at least $4$.
\end{itemize}

Next, we make seven copies $G_{1}^{(j)}$, $j=0,\ldots,6$, of $G_{1}$ and we denote the corresponding subpaths of the boundary by $q_{1}^{(j)},\ldots,q_{4}^{(j)}$, $j=0,\ldots,6$. We paste:
\begin{itemize}
  \item $G_{1}^{(0)}$ to $G_{1}^{(1)}$  along $q_{1}^{(0)}$ and $q_{1}^{(1)}$ (after reflecting $G_{1}^{(1)}$),
  \item $G_{1}^{(0)}$ to $G_{1}^{(2)}$ along $q_{2}^{(0)}$ and $q_{2}^{(2)}$ (after reflecting $G_{1}^{(2)}$),
  \item $G_{1}^{(1)}$ to $G_{1}^{(2)}$ along $q_{2}^{(1)}$ and $q_{1}^{(2)}$ (which is possible as  $q_{2}^{(1)}$ and $q_{1}^{(2)}$ both originate from $p_{1}$),
  \item $G_{1}^{(2)}$ to $G_{1}^{(3)}$ along $q_{3}^{(2)}$ and $q_{3}^{(3)}$ (after rotating $G_{1}^{(3)}$),
  \item $G_{1}^{(0)}$ to $G_{1}^{(4)}$ along $q_{3}^{(0)}$ and $q_{3}^{(4)}$ (after reflecting $G_{1}^{(4)}$),
  \item $G_{1}^{(3)}$ to $G_{1}^{(4)}$ along $q_{2}^{(3)}$ and $q_{2}^{(4)}$,
  \item $G_{1}^{(0)}$ to $G_{1}^{(5)}$ along $q_{4}^{(0)}$ and $q_{4}^{(5)}$ (after reflecting $G_{1}^{(5)}$),
  \item  $G_{1}^{(4)}$ to $G_{1}^{(5)}$ along $q_{4}^{(4)}$ and $q_{3}^{(5)}$ (which is possible as  $q_{4}^{(4)}$ and $q_{3}^{(5)}$ both originate from $p_{2}$),
  \item $G_{1}^{(5)}$ to $G_{1}^{(6)}$ along $q_{1}^{(5)}$ and $q_{1}^{(6)}$ (after rotating $G_{1}^{(6)}$),
  \item  $G_{1}^{(1)}$ to $G_{1}^{(6)}$ along $q_{4}^{(1)}$ and $q_{4}^{(6)}$
\end{itemize}

We denote the resulting graph by $G_{2}$  and the boundary path $q_{1}^{(3)}+ q_{1}^{(4)}+ q_{2}^{(5)}+ q_{2}^{(6)}$ by $r_{1}$ and  $q_{3}^{(6)}+ q_{3}^{(1)}+ q_{4}^{(2)}+ q_{4}^{(3)}$ by $r_{2}$. We summarize some facts about  the vertex degrees in $G_{2}$:
\begin{itemize}
  \item  All vertices in $G_{2}$ that are not in the boundary path of $G_{2}$ but were in the boundary paths of $G_{1}^{(0)},\ldots,G_{1}^{(6)}$ have vertex degree at least $6$. (The inner vertices of $q_{1}^{(j)},\ldots q_{4}^{(j)}$, $j=0,\ldots,6$, had vertex degree at least $4$ before being pasted. The outer vertices of $q_{1}^{(j)},\ldots q_{4}^{(j)}$, which are originating from the vertices $u$, $u'$ and $v$, had vertex degree at least $3$ before and each is pasted to at least three copies.)
  \item The vertex in the boundary path in the intersection of $q_{1}^{(3)}$  and $q_{4}^{(3)}$ from $G_{1}^{(3)}$ (which originated from a copy of $v$) has vertex degree at least $3$. The same applies to the vertex in the intersection of $q_{2}^{(6)}$  and $q_{3}^{(6)}$ from $G_{1}^{(6)}$.
  \item All other vertices in the  boundary path of $G_{2}$ have vertex degree at least $4$.
\end{itemize}

Next, we make four copies $G_{2}^{(1)},\ldots, G_{2}^{(4)}$ of $G_{2}$. We paste:
\begin{itemize}
  \item $G_{2}^{(1)}$ to $G_{2}^{(2)}$ along $r_{2}^{(1)}$ and $r_{2}^{(2)}$  (after reflecting $G_{2}^{(2)}$),
  \item $G_{2}^{(2)}$ to $G_{2}^{(3)}$ along $r_{1}^{(2)}$ and       $r_{1}^{(3)}$,
  \item  $G_{2}^{(3)}$ to $G_{2}^{(4)}$ along $r_{2}^{(3)}$ and $r_{2}^{(4)}$ (after reflecting $G_{2}^{(4)}$)
  \item  after embedding the resulting graph into the two dimensional sphere $\mathbb{S}^{2}$ we paste $G_{2}^{(1)}$ to $G_{2}^{(4)}$ along $r_{1}^{(1)}$ and $r_{1}^{(4)}$.
\end{itemize}
The resulting graph $G_{3}=(V_{3},E_{3})$ is a planar graph that can be  embedded in  the sphere $\mathbb{S}^{2}$. By the Gau\ss-Bonnet formula, see e.g. \cite[Proposition~1]{K2} (or \cite{BP1} for tessellations),
\begin{align*}
\sum_{v\in V_{1}} \ka^{(G_{3})}(v)=\chi(\mathbb{S}^{2})=2,
\end{align*}
where $\chi $ denotes the Euler characteristic which equals $2$ for the sphere. Denote the vertices in $G_{3}$ that  result from {copies of }vertices in $V(p)$ of the original subgraph $G'$ by $V_{3}^{(p)}$. By our construction they
have degree at least $6$ in $G_{3}$. Thus, $\ka^{(G_{3})}(v)\leq0$ for any such vertex $ v\in V_{3}\setminus V_{3}^{(p)}$. (Note that the minimal face degree is $3$ due to triangles and, therefore, $ \ka_{C}^{(G_{3})}(v,f)\leq 1/6-1/2-1/3=0 $ which implies $\ka^{(G_{3})}(v)\leq0$).
On the other hand, for every vertex $v'\in V'\setminus V(p)$ there are $ 56 =2\cdot 7\cdot 4$ copies in  $V_{3}\setminus V_{3}^{(p)}$ and for any such copy $v$ of $v'$ we have $\ka^{(G_{3})}(v)=\ka^{(G')}(v)=\ka^{(G)}(v)$.
In conclusion, we have
\begin{align*}
    2=\sum_{v\in V_{3}} \ka^{(G_{3})}(v)\leq \sum_{v\in {V_{3}}\setminus V_{3}^{(p)}} \ka^{(G_{3})}(v)=56\sum_{v\in V'\setminus V(p)} \ka^{(G')}(v),
\end{align*}
which implies the statement.
\end{proof}

There is an immediate corollary which will plays a major role in the considerations  below.
\begin{cor}\label{c:CopyPaste} Let $G'=(V',E')$ be a subgraph of a planar graph  $G=(V,E)$ with a simply closed boundary path such that every interior vertex has degree larger or equal to $6$ in $G'$ and all but two vertices in the boundary path have degree at least $3$. Then, there are at least four vertices in the boundary path that have vertex degree at most $3$ in $G'$.
\end{cor}
\begin{proof}Suppose 
	 all but three vertices  in the boundary path have degree larger or equal to $4$. Then, the assumptions of the lemma above are fulfilled and, therefore, there exists a vertex in the interior with positive curvature. This however is impossible by the assumption that the vertex degrees of the interior vertices are larger or equal to $6$.
\end{proof}

%%%%%%%%%%%%%%%%%%%%%%%%%%%%%%%%%%%%%%%%%%%%%%%%%%%%%%%%%%%
\section{Triangulations}\label{s:triangulation}

In this section we  begin by proving Theorem \ref{t:main}
 and its consequences in the simpler case of triangulations.
  The case of general planar graphs will be  investigated in Section \ref{s:general-planar}.

To phrase the result in the  special case of triangulations we need some notation first.
  	We denote $\mathcal B_r$ the embedding of the vertices and the edges of the distance balls $B_r$, confer Section~\ref{s:setup}, into $\mathcal S$.
  	Since $\mathcal B_r$ is a  compact set, $\mathcal S \setminus \mathcal B_r$ admit only one unbounded connected component  that we denote by $\mathcal U_r$. We also denote $$ U_r:=  V \cap \mathcal U_r. $$
Observe that in general $ V $ only strictly includes the union $ B_{r}\cup U_{r} $ as there can be vertices  in $ \mathcal{B}_{r} $ which distance larger than $ r $ from the root which are however enclosed by vertices in $ B_{r} $.

\begin{thm}\label{t:triangulation} Let $G=(V,E)$ be a planar triangulation such that  one of the following assumptions holds:
\begin{itemize}
	\item [(a)] $\deg\ge 6$ outside of the root $o$.
	\item [(b)]
	  $\deg\ge 7$ outside of $B_{r}$ for some $r\ge0$. 
\end{itemize}
Let	   $v \in S_R\cap U_{r}$ with $R>r+  \log_{2}|S_{r}|$, (where $ r=0 $ in case $ \mathrm{(a)} $). Then,
	\begin{align*}
	\deg_{0}(v)=2  \qquad\mbox{and}\qquad  1\leq    \deg_{-}(v)\leq 2.
	\end{align*}
In particular, any geodesic reaching such a vertex $ v $ from $ o $ can be continued indefinitely.
Furthermore, all spheres  $ S_{R}\cap U_{r} $ are given by a cyclic path.
\end{thm}

\begin{rem}In Section~\ref{s:othercomp} we show that we can extend our techniques so that the results also hold for all $ v\in S_{R} $ instead of $ S_R\cap U_{r} $. However, the result as stated above is enough to prove our main results.
\end{rem}

To prove the theorem we employ the copy and paste procedures above. But before we have to introduce a new sphere structure that takes heed {of} the fact  that geodesics might not be continued to infinity.

%%%%%%%%%%%%%%%%%%%%%%%%%%%%%%%%%%%%%%%%%%%%%%%%%%%%%%%%%
\subsection{A new sphere structure}\label{s:sphere-structure}
We first introduce some notation.
Recall that for  each simply closed path $p$ in the graph induces a simply closed curve with image $ \gamma(p) $ in the surface $\mathcal{S}$ where the graph is embedded. Furthermore, recall that Jordan's curve theorem induces a partition of $\Sc$  as
$$ \Sc=\Bc (p)\cup \gamma(p) \cup \Uc(p), $$  with $\Bc(p)$  and $\Uc(p)$ being respectively the bounded and the unbounded connected component of $\Sc \setminus \gamma(p)$. % We also set  $\Cc(p):=\Bc(p)\cup \gamma(p) =\overline{\Bc(p)}$

We also denote $$ B(p):=V\cap \overline{\Bc(p)}=V\cap(\mathcal{B}(p)\cup\gamma(p))\;\;\mbox{ and }\;\;U(p):=V\cap \Uc(p).$$
Furthermore, recall that we denoted $ V(p):=V\cap \gamma(p)  $.

\begin{lemma}\label{l:decomposition} 
Let $ G $ be planar triangulation and $r\geq 0$.   The subgraph  induced by  $U_r$  is connected and infinite. Moreover  if $w \sim w'$ with $w\in U_r$ and $w'\notin B_r$, then $w'\in U_r$.
\end{lemma}

%\red{Actually we use more the reciprocal, if $v \sim w$  and $v\in U_r$ and $w\notin B_r$, thus $w\in U_r$.}
\begin{proof} First since $\Uc_r$ is unbounded, the fact that the subgraph induced by $G$  on $U_r$ is infinite is clear.
Let $ v,w\in U_r $. %We show that $ v $ and $ w $ lie in the same connected component of $ V\setminus B_{r} $. 
First, let $ \gamma  $ be a simply closed curve in $ \mathcal{S} $ such that $ v,w $ are on $ \gamma $ and $ B_{r}  $ lies in the open bounded region enclosed by $ \gamma $. Let $ f_{1},\ldots,f_{n} $ be a path of faces which $ \gamma $ 	passes through from $ v $ to $ w $, i.e., two subsequent faces intersect in exactly one edge which is crossed by $ \gamma $.  These edges have at least one vertex outside of $ B_{r} $ which we denote by $ v_{1},\ldots,v_{n} $. Two subsequent vertices $ v_{j} $ and $ v_{j+1} $ are connected by a path of boundary vertices of $ f_{j} $ which are not included in $ B_{r} $, $ j=1,\ldots,n-1 $. This induces a path in the graph between $ v $ and $ w $ in $U_r $. %Thus, $ v $ and $ w $ are in the same connected component of $ V\setminus B_{r} $.
The last statement is easy: indeed if  $w \sim w'$, $w\in U_r$ and $w'\notin B_r$, the edges joining $w$ to $w'$ gives a curve in $\mathcal S \setminus \mathcal B_r$ . Thus $w$ and $w'$ are in the same connected component of  $\mathcal S \setminus \mathcal B_r$ and the statement follows.
\end{proof}

Note that since the other connected components of $\mathcal S \setminus \mathcal B_r$ {are bounded}, the other induced graphs are finite.

%\blue{I do not think that we use the fact that $U_r$ is connected. In fact yes,  just for  unbounded!} 

The following definition turns to be an  important object in our study. 
\begin{equation*}%\label{eq:def-Vr}
V_{r}:=\{v\in B_{r}\mid  \mbox{there is }w\in U_{r}\mbox{ such that }v\sim w  \} . 
\end{equation*}

\begin{lemma}\label{l:V_r}Let $ G $ be planar triangulation. For every  {$ r\ge 1$}, there exists a simply closed path of vertices $ p_{r} $ such that $ V_r=V(p_{r}) $.  
Moreover, one has 
\begin{align*}
B_{r} \subseteq  B(p_r), \quad    \textrm{ and }\quad
 U_r=U(p_r)\subseteq V\setminus B_r.
\end{align*}
Furthermore, one also has
\begin{align*}
V_{r}&=\{v\in S_{r}\mid  \mbox{there is }w\in U_{r}\mbox{ such that }v\sim w  \} \\ 
&=\{v\in B (p_ r)\mid  \mbox{there is }w\in U_{r}\mbox{ such that }v\sim w  \}.
\end{align*}
\end{lemma}
\begin{proof}
\emph{We show $\emptyset\neq V_r \subseteq S_r$:}
Since  the graph is connected, $V_r$ is not empty.   Moreover, by construction,  $U_r \subseteq V\setminus B_r$. So, for a vertex in $ V_{r} $ to be connected to $ U_{r} $ it cannot be in $ B_{r-1} $. Thus, we have  $V_r \subseteq (V\setminus B_{r-1}) \cap B_r= S_r$.\smallskip

\emph{Existence of a simply closed path $ p_{r} $ with $ V(p_{r})\subseteq V_{r} $:}
We now claim that for every vertex in $ v\in V_{r} $,	 there are (at least) two distinct adjacent vertices in $ V_{r} $. This easily follows by considering the combinatorial neighborhood of $ v  \in S_{r}$ which includes  a neighbor  $ v_{-} $ in $ S_{r-1} $ and $ v_{+ } $ in $ S_{r+1}  \cap U_r $   and using that $ G $ is a triangulation.
Since $S_r$ is finite, there necessarily exists a simply closed path $ p_{r} $ of vertices in $ V_{r} $. \smallskip

\emph{We show that the root $o $  is in  $  B(p_{r}) $:}  By construction of the simply closed path $ p_{r} $, and its image $ \gamma(p_{r}) $ which is a simply closed curve in $\mathcal{S}  $, there exist two  vertices $v_{-}\in S_{r-1}$ and $v_{+}\in S_{r+1} \cap U_r$ which do not belong to the same connected component of $\Sc  \setminus\gamma( p_r)$.   Since $V(p_r)\subseteq S_r\subseteq B_r$, one has $\mathcal U_r\subseteq \mathcal U(p_r)$. Thus, $ v_{+}\in \mathcal{U}(p_{r}) $ and $ v_{-} \in \mathcal{B}(p_{r}) $. Furthermore, all the vertices on the  geodesic between  the root $o$ and $v_{-}$ belong to the same  connected component and, therefore, belong  to $B(p_r)$.\smallskip

\emph{We show $U(p_r)\subseteq V\setminus B_r$ and $B_r \subseteq B(p_r)$:} The two  statements are equivalent by taking   the complement. Let us prove the first statement.  Let $v \in U(p_r)$ and consider a  geodesic from $o$ to $v$. By connectedness, it has to cross the simply closed curve $\gamma(p_r) $ arising from $ p_{r} $ and must contain a vertex  in $V(p_r)$, necessarily different from $v$. Since $V(p_r)\subseteq S_r$, one has  $d(o,v) \geq r+1$. Thus, $ v $ is not in $ B_{r} $.\smallskip

\emph{We show $V_r= V(p_r)$:} Since we constructed $ p_{r} $ such that $ V(p_{r})\subseteq V_{r} $, we are left to show $V_r\subseteq  V(p_r)$.  Let $v\in V_r$ and $w \in U_r $ such that $v\sim w$.
Consider a geodesic $ p $ from $ o $ to $ v $. Thus, adding $ w $ to the end of $ p $ is a geodesic between $o$ and $w$ which,  by connectedness, must cross $V(p_r)$, since $o \in B(p_r)$ and $w\in U_r \subseteq  U(p_r)$. Recalling $V(p_r) \subseteq S_r$, this implies  $v\in V(p_r)$.\smallskip

\emph{We show $ U_r=U(p_r)$:}  We already noticed that $U_r \subseteq U(p_r)$. The reverse   inclusion  follows from the following general  connectedness result: Let $A \subseteq B \subseteq E$ in a topological space $ E $ and let $O_A$ be an arc-connected component of $E\setminus A$ such that $O_A\subseteq E\setminus B$. Then, $O_A$ is also an arc-connected component of $E\setminus B$.\smallskip

Finally, we turn to the last two equalities concerning $V_r$.
The first  equality is clear, since we already know $V_r\subseteq S_r$.
We turn to the second equality.
Let $v\in  B(p_r)$ and $w \in U_r=U(p_r)$ such that $v\sim w$.
Let   $p$ be any path form $o$ to $v$.
Since $o\in  B(p_r)$, $w\in  U(p_r)$ and $ v\sim w $, there is $u\in V(p)$   such that $u\in V(p_r)$. Consider the last vertex $ u $ in the path $ p_{r} $ with this property. Then this vertex and all the following vertices including $ v $ must be in $ V(p_{r})\cup U(p_{r}) $. Since $ v \not\in U(p_{r}) $, we conclude $v\in V(p_r)=V_r.$
\end{proof}

We define inductively $\Sigma_{0}=S_{0}=\{o\}$ and
\begin{align*}
    \Sigma_{r}:={B}(p_{r})\setminus \Sigma_{r-1}\quad \mbox{and}\quad \partial \Sigma_{r}:=V(p_{r}),\qquad r\ge1.
\end{align*}
This  gives a decomposition of $ V $ into a ``new sphere structure''. This new sphere structure is also called a  ``1-dimensional decomposition'' in the literature. 
We  denote 
$$ B_r^{(\Sigma)}:= \bigcup_{k=0}^r \Sigma_{k}$$ for  $r\geq0$.
{Note that $B_r^{(\Sigma)}= B(p_r)$ for $r\geq 1$.}

\begin{lemma}\label{l:sphere-structure} Let $G$ be a planar triangulation. Let  $ r,r' \geq 0$ and $v \in  \Sigma_{r}$, $w \in  \Sigma_{r'}$ such that $v \sim w$, 
then
\[
| r'-r | \leq 1.
\]
\end{lemma}
\begin{proof}%We use the notation of Lemma \ref{l:V_r}.
Let  $r\geq0$, $l\geq1$ and   $v \in  \Sigma_{r}$, $w \in  \Sigma_{r+l}$ such that $v \sim w$.
Since by construction $ B(p_r) = B_r^{(\Sigma)}$, we deduce  $w\in U(p_r)=U_r$ (where $ p_{r} $ is taken from Lemma~\ref{l:V_r}).
By definition of $ V_{r} $ and $ v\sim w $, we  obtain $v\in V_r=\partial \Sigma_r \subseteq S_r$ and, therefore, $w\in S_{r+1}$. Hence, $w\in B_{r+1}^{(\Sigma)}$ as   $B_{r+1} \subseteq V \cap B(p_{r+1}) = B_{r+1}^{(\Sigma)}$.  Thus,  $w \in  \Sigma_{r+1}$ that is $l=1$ and the result follows.
\end{proof}

Next, we analyze this new sphere structure more closely. To this end, we denote by $\deg_{\pm}^{(\Sigma)}(v)$ (respectively $\deg_{0}^{(\Sigma)}(v)$) the number of neighbors of a vertex $v\in \Sigma_{r}$ in $\Sigma_{r\pm1}$ (respectively in $\Sigma_{r}$).

\begin{lemma}\label{l:sphere} In a planar triangulation we have 
 $$ \partial \Sigma_r=V_{r},$$
 for all  $ r\ge1 $. Moreover, on $  \partial \Sigma_r$,  $$\deg_{+}\ge \deg_{+}^{(\Sigma)} \geq 1,\quad \deg_-=\deg_{-}^{(\Sigma)}\ge 1 \quad \mbox{and}\quad2\leq \deg_{0}\leq \deg_{0}^{(\Sigma)}$$ 
and, on $ \Sigma_r \setminus \partial \Sigma_r$,
\[
\deg_+^{(\Sigma)} =0.
\]
%Moreover,    every geodesic starting at $ o $ and passing through $ \Sigma_{r} $ can be continued indefinitely if and only if $ \Sigma_{R}=\partial \Sigma_{R} $ for all $ R\ge r $.% Furthermore, if $ \Sigma_{r}=\partial \Sigma_{r} $  for all $ r\ge 1 $, then $ \Sigma_{r}=S_{r} $ for all $ r\ge 1 $.\Hmm{Added the last statement as it used in the proof of Theorem~\ref{t:triangulation}}
\end{lemma}
\begin{proof}The first statement $ \partial \Sigma_{r}=V_{r}  $  follows directly from Lemma \ref{l:V_r} and the definition of $\partial \Sigma_{r}$. We first consider the case of $ v\in \partial \Sigma_{r}=V_{r} $.

\emph{ We  show $ \deg_{+}(v)\ge \deg_{+}^{(\Sigma)}(v) \geq 1 $ for  $ v\in \partial \Sigma_{r} $:} We first claim
	\[
\{ w \in S_{r+1}\mid w \sim v \}  \supseteq \{ w \in  \Sigma_{r+1}\mid w \sim v \}\neq \emptyset. 
\]
Vertices in $ V_{r } $ have (at least) one neighbor in  $U_{r}=U(p_{r})=V\setminus {B}(p_{r})$.  By definition of $ \Sigma_{r+1} $  these neighbors are exactly the neighbors of $v$ in $ \Sigma_{r+1}$ and we have already seen in Lemma~\ref{l:V_r} that  they must belong to  $ S_{r+1} $.  
Thus,	we   obtain 	 $  \deg_{+}(v)  \geq\deg_{+}^{(\Sigma)}(v) \ge1 $.

\emph{ We  show $\deg_-(v)=\deg_{-}^{(\Sigma)}(v)\ge 1$ for $ v\in \partial \Sigma_{r} $:}	To this end, we claim 
\[
\{ w \in \Sigma_{r-1}\mid w \sim v \} =\{ w \in \partial \Sigma_{r-1}\mid w \sim v \}=\{ w \in S_{r-1}\mid w \sim v \} \neq \emptyset. 
\]
To see this, we first note that by connectedness $\partial \Sigma_{r}  \subseteq U_{r-1}$.  Thus,  if $ w\in S_{r-1} $ is a neighbor of $ v\in \partial \Sigma_{r}$, then $ w\in V_{r-1} = \partial \Sigma_{r-1}$ by definition and, therefore, $ w\in  \Sigma_{r-1}$. On the other hand,  if $ w\in \Sigma_{r-1} $ is a neighbor of $ v\in \partial \Sigma_{r} $, then it must belong to $\partial \Sigma_{r-1}$ and as noticed before to $S_{r-1}$.

% $ \Sigma_{r-1}\subseteq S_{r-1} $ and we conclude $ \deg_{-}^{(\Sigma)}(v) =\deg_{-}(v)$ for $ v\in\partial \Sigma_{r} $. The inequality $ \deg_{-}\ge 1 $ is obvious from the definition of the spheres.
	
\emph{We show $ 2\leq \deg_{0}\leq \deg_{-}^{(\Sigma)} $ $ v\in \partial \Sigma_{r} $:}	The second inequality  follows from the first two inequalities. Furthermore, $ v\in \partial\Sigma_{r}=V_{r} $ has two neighbors in $V_{r}= V(p_{r})\subseteq S_{r}$.  %\Sigma_{r}\cap S_{r} $.
	
\emph{We show $ \deg_{+}^{(\Sigma)} (v)=0 $ for $ v\in  \Sigma_r \setminus \partial \Sigma_{r}$:} If $ v\in \Sigma_r \setminus \partial \Sigma_{r}$,   then $v$ does not admit any neighbor in $U_r$. Thus, $ v $ has no neighbors  in $\Sigma_{r+1}$  since $\Sigma_{r+1}\subseteq U_r$ from which $ \deg_{+}^{(\Sigma)} (v)=0$ follows.  

This finishes the proof.
%We now turn to the Cartan-Hadamard type result on geodesics. Note that all geodesic reaching $U_r$ can be continued indefinitely if and only if $\deg_+(v)\geq 1$ for all $v\in U_r$. 
%First we assume that $ \Sigma_{R}=\partial \Sigma_{R} $ for all $R\geq r$. This gives  $\Sigma_R=S_R \cap U_{r-1}$ for all $R\geq r$. Therefore on the connected component $U_{r-1}$ we have $\deg_+= \deg_+^{(\Sigma)}\geq 1$ and the conclusion follows.
%Now, assume that there is a vertex $ \Sigma_{l} \setminus \partial \Sigma_{l}\neq \emptyset $ for some $l\geq r$. We have $d(o,\cdot)\geq l$ on $  \Sigma_{l} $ and by local finiteness and the function $d(o,\cdot)$ assumes its  maximum  on $ B_l^{(\Sigma)} \cap U_r$ in a point $v$ such that $d(o,v)\geq l$. Since $v$ does not admit any neighbor in $V\setminus B_l^{(\Sigma)}$,  one infers that $\deg_+(v_0)=0$.
%Finally, we show that  $\Sigma_r= \partial \Sigma_r $ implies $\Sigma_r=S_r $ for all $ r\ge 1 $. By definition of $ \partial \Sigma_{r}=V_{r} $ we obtain $ \Sigma_{r}  \partial \Sigma_{r}=V_{r} \subseteq S_{r}$. Now, the statement follows by induction. Assume $ \Sigma_{r-1}=S_{r-1} $ and $ v\in S_{r} $. Then, there is $ w\in S_{r-1} $ such that $ v\sim w $. By Lemma~\ref{l:sphere-structure}  we infer $ v\in \Sigma_{r} $.
\end{proof}

\subsection{Elementary cells} \label{s:elementary-cells}

In this section we define  elementary cells associated to the  new sphere structure introduced above. More precisely, we define the elementary cells $C_{v,w}$ and $C_{v}$ associated to vertices $v,w\in \partial \Sigma_{r}$ with $v\sim {w}$.

Let $E_{r}$ be the {set of} edges $e_{1},\ldots, e_{N}$ connecting vertices in $\partial \Sigma_{r}$ to vertices $\partial \Sigma_{r-1}$, where the enumeration is in cyclic order with respect to the boundary path $ p_{r} $ from Lemma~\ref{l:V_r}. In particular, each vertex in $\partial \Sigma_{r}$ is contained in at least one of these edges by the definition of $ V_{r}$  which equals $V(p_{r}) =\partial \Sigma_{r}$ by the lemmas above. Now, the subgraph ${(}\Sigma_{r}\setminus \Sigma_{r-1}{)}\cup \partial \Sigma_{r-1}$ can be decomposed into $N$ subgraphs $W_{1},\ldots, W_{N}$ that  have a closed boundary path with edges of $\partial \Sigma_{r}$,  $\partial \Sigma_{r-1}$ and $E_{r}$ and such that $W_{j}$ and $W_{j+1}$ intersect precisely in $e_{j}$ for $j=1,\ldots, N$ (modulo $N$). Note that each $W_{j}$ contains exactly one or two vertices of $\partial \Sigma_{r}$.

If there are two vertices $v,w\in\partial \Sigma_{r}$ contained in $W_{j}$ we denote $C_{v,w}:=W_{j}$ and call $C_{v,w}$ an \emph{elementary cell} of type (EC1). We denote the neighbors of $v$ and $w$ in $\partial \Sigma_{r-1}$ by the edges $e_{j}$ and $e_{j+1}$ by $v'$ and $w'$ (while it might very well happen that $v'=w'$).

On the other hand, if there is only one vertex  $v\in \partial \Sigma_{r}$ contained in $ W_{j} $, then $ v $ is contained in more than one edge of $E_{r}$, say $e_{i},\ldots,e_{i+n}$. We denote the union of $W_{i+1},\ldots,W_{i+n}$ by $C_{v}$ and call $C_{v}$ an \emph{elementary cell} type (EC2).

See Figure~\ref{f:Elementary_Cell} for an illustration of the definition.

%An \emph{elementary cell} is a subgraph $(W,E_{W})$ of $\Sigma_{r}\cup \partial \Sigma_{r-1}$ with a closed boundary path $p$ that is is completely contained in $\partial \Sigma_r\cup\partial \Sigma_{r-1}$ such that one of the following two assumptions is satisfied:
%\begin{itemize}
  %\item [(EC1)] there are two adjacent vertices $v,w\in W\cap \Sigma_{r}$ of $p$  and at most two vertices $v',w'\in\partial W\cap \Sigma_{r-1}$  in $p$ such that $v'$ is the only neighbor of $v$ in $W\cap \Sigma_{r-1}$ and $w'$ is the only neighbor of $w$ in $W\cap \Sigma_{r-1}$ or
  %\item  [(EC2)] there is only one vertex $v$ of $p$ in  $W\cap \partial \Sigma_{r}$ and every neighbor of $v$ in $\partial \Sigma_{r-1}$ is in $p$.
%\end{itemize}
%See Figure~\ref{f:Elementary_Cell} for illustration. Since every vertex in $\partial \Sigma_{r}$ has a neighbor in $\partial \Sigma_{r-1}$, any two vertices $v,w\in\partial \Sigma_{r}$, $v\sim w$, induce a unique elementary cell satisfying (EC1) which we denote by $C_{v,w}$. The vertices $v',w'\in\partial \Sigma_{r-1}$ from (EC1) are then uniquely determined, while it might very well happen that $v'=w'$. On the other hand, every vertex $v\in\partial \Sigma_{r}$ with $\deg_{-}(v)\ge2$ induces an elementary cell satisfying (EC2) which we denote by $C_{v}$. For $v\in\partial \Sigma_{r}$ with $\deg_{-}(v)=1$, i.e., there is a is a single vertex $w\in \partial \Sigma_{r-1}$ with $v\sim w$, we let  $C_{v}$ be the subgraph induced by $v$ and $w$, which consists of only one edge.

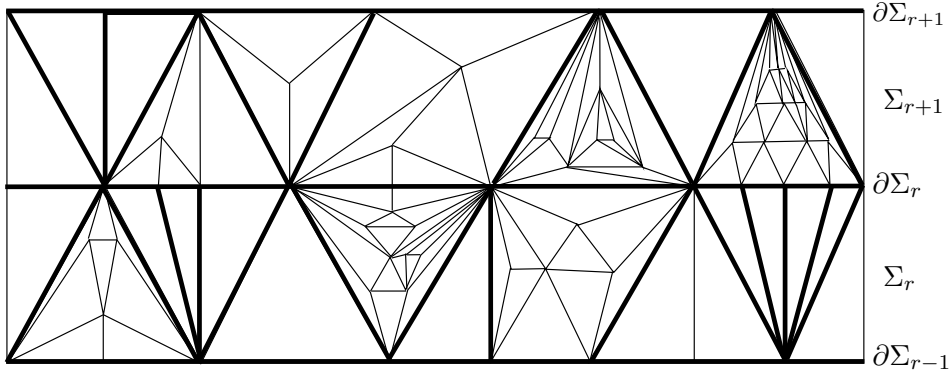
\begin{figure}[!h]
\begin{overpic}
	[width=1\textwidth]{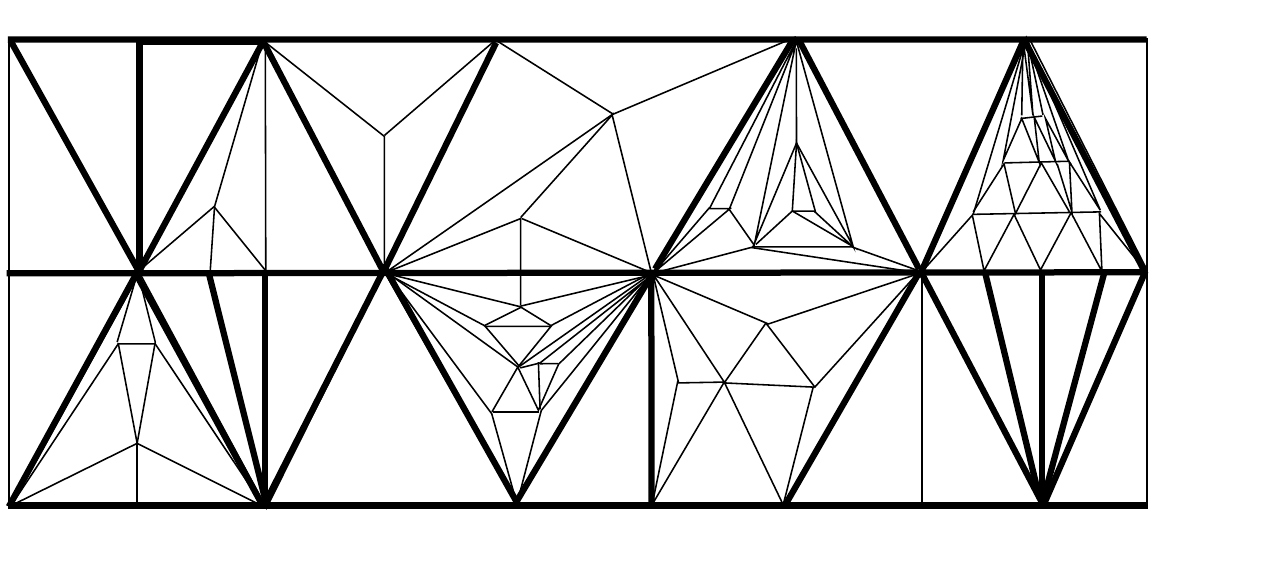}
	\put(330,145){$\partial \Sigma_{r+1}$} 	\put(334,112){\large$\Sigma_{r+1}$}
	\put(330,80){\large$\partial\Sigma_{r}$}
	\put(334,45){\large$\Sigma_{r}$}
	\put(330,15){\large$\partial\Sigma_{r-1}$}
\end{overpic}	
\caption{\label{f:Elementary_Cell} An example of a part of $\Sigma_{r}$ and $\Sigma_{r+1}$, where the horizontal lines indicate the boundary edges connecting $\partial \Sigma_{r-1}$, $\partial \Sigma_{r}$ and $\partial \Sigma_{r+1}$. The thick lines enclose the elementary cells.}
\end{figure}

%Next, we demonstrate that $\Sigma_{r}\cup\partial \Sigma_{r-1}$ can be represented as a cycle of elementary cells.

%\begin{lemma} Let $G=(V,E)$ be a triangulation. Then, the subgraph induced by $\Sigma_{r}\cup \partial \Sigma_{r-1}$ is the union of  $C_{v_{0}}, C_{v_{0},v_{1}}, C_{v_{1}}\ldots,C_{v_{n-1}},C_{v_{n-1},v_{n}},C_{v_{n}}, C_{v_{n},v_{0}}$, where $(v_{0},\ldots,v_{n})=p_{r}$. \end{lemma}
%\begin{proof} By definition, every vertex of these elementary cells is contained in $\Sigma_{r}\cup \partial \Sigma_{r-1}$. On the other hand, let $v,w\in\partial \Sigma_{r}$   and note that the graphs $C_{v}$ and $C_{v,w}$ share exactly one edge of their boundary paths. In this way we see the other inclusion. \end{proof}

%%%%%%%%%%%%%%%%%%%%%%%%%%%%%%%%%%%%%%%%%%%%%%%%%%%%%%%%%
\subsection{The induction and  the proof  for triangulations}

In this section, we give the proof of Theorem \ref{t:triangulation} which deals with the case of triangulations.  
Our strategy is to show that if a triangulation has large vertex degree (outside a finite set)   then all the elementary cells (at least outside some other finite set) must have empty interior.

This will yield the results on the degree and  the Cartan-Hadamard type result. Moreover, this will also give that the graph has a  nice structure since (at least outside some  finite set) the sphere are composed exactly by a  cyclic closed path.

We introduce the set $Z_{R}\subseteq  \partial \Sigma_{R}\subseteq  S_{R}$ by
\begin{equation*}
    Z_{R}:=\{v\in \partial \Sigma_{R}\mid \deg_{+}^{{(\Sigma)}} (v)=1\},\quad R\ge0.
\end{equation*}
Our strategy is to show that if $Z_R$ is non-empty for some $ R $ then $ Z_{R} $ grows as $R$ decays. The underlying idea is that vertices in $ Z_{R} $ have a lot of ``backward'' neighbors due to the large vertex degree which then inductively yields vertices in $ Z_{R-1} $.

 We first   consider  the case $ \deg\ge 6 $ outside  of a finite set.   The case $ \deg \ge 7 $ outside of a finite set will need somewhat more effort.
%which we show to grow as $r$ decays.% whenever the backward degree of a vertex is  larger than $2$ or the forward degree is less than $2$.
\subsubsection{The case $\deg \geq 6$.}

\begin{lemma}\label{l:emptyinterior} Let $G=(V,E)$ be a triangulation such $\deg\ge 6$ outside of $\Sigma_{r}$ for some $r\ge0$. If, there are $v,w\in \partial \Sigma_{R}$, $v\sim w$, $R>r$, such that one of   the induced elementary cells $C_{v}$ and $C_{v,w}$  does not have an  empty interior, then  this elementary cell contains   points in $ Z_{R-1}$ and
\begin{align*}
    Z_{R-1}\neq \emptyset.
\end{align*}
In particular, if there is  $R>r$ such that $\Sigma_R \setminus \partial    \Sigma_R \neq \emptyset$, then $ Z_{R-1}\neq \emptyset$.
\end{lemma}
\begin{proof} Let  $v,w\in \partial \Sigma_{R}$, $v\sim w$, and  consider the elementary cell $C_{v,w}$. Assume the interior of the elementary cell $C_{v,w}$  is non-empty.
Since we are in a triangulation, non-emptiness of $C_{v,w}$ implies that each of the vertices $v,w,w',v'$ has vertex degree at least $3$ in $C_{v,w}$. Moreover, by Lemma~\ref{l:sphere} any other boundary vertex $ u $ of $ C_{v,w} $ (in $ \partial \Sigma_{R-1} $) has $ \deg_{+}^{(\Sigma)}(u)\ge1$. So, every vertex in the boundary path of $C_{v,w}$ has degree at least $3$ in $C_{v,w}$. We distinguish two cases:\smallskip

\emph{Case 1: One of the vertices $v,w,w',v'$ has degree at least $4$  in $C_{v,w}$.}
Since every vertex in the boundary of $C_{v,w}$ has degree at most $3$ and every interior vertex of $C_{v,w}$ has degree at least $6$ by assumption, we can apply Corollary~\ref{c:CopyPaste}. This yields that there are at least four vertices in the boundary with degree at most $ 3 $. Since at least of the vertices $v,w,w',v'$ has degree $ 4 $   there is another boundary vertex with degree at most $3$. Since the only vertices in $\partial \Sigma_{R}$ are $v,w$ this vertex is in $\Sigma_{R-1}$ and, therefore,  this vertex is in $Z_{R-1}$.\smallskip

\emph{Case 2: The vertices $v,w,w',v'$ have all degree $3$  in $C_{v,w}$.} Since we are in a triangulation, the vertices share a unique common neighbor which we denote by $u$.
The subgraph $G'$ induced by $C_{v,w}\setminus \{v,w\}$ has the boundary $p'=(w',u,v')+ q$, where $q$
is the subpath of the boundary path $p$ of $C_{v,w}$ such that $p=(w',w,v,v')+ q$. Note that by the assumption $\deg\ge6$, we infer that $u$ has degree at least $4$ in $G'$. If every inner vertex in the subpath $q$  has vertex degree at least $4$ in $C_{v,w}$, then $G'$ satisfies the assumption of the Copy-and-Paste Lemma, Lemma~\ref{l:CopyPaste}.
Hence, there is an interior vertex with positive curvature in $G'$. This is, however, impossible by the assumption $\deg\ge6$. Thus, there is at least one inner vertex in the boundary path $q$ that has vertex degree strictly less than $4$ in $C_{v,w}$. Since all vertices in $q$ are in $\partial \Sigma_{R-1}$, they have vertex degree exactly 3 and, thus, this vertex is in $Z_{R-1}$.\smallskip

Consider now the elementary cell $C_{v}$. If the interior is not empty, then $v$ has vertex degree at least $3$ with in $C_{v}$. By the  Copy-and-Paste Lemma, Lemma~\ref{l:CopyPaste} and a similar argument as above this implies that there is a vertex in $Z_{R-1}$.\smallskip

The ``in particular'' is clear since if $v\in \Sigma_R \setminus \partial    \Sigma_R \neq \emptyset$, then $ v $ must belong to the interior of some elementary cell.
\end{proof}

%%%%%%%%%% from here %%%%%%%%%%%%%%%%

\begin{lemma}[The base case]\label{l:base_case} Let $G=(V,E)$ be a triangulation with $ \deg\ge6 $ outside of $ S_{r} $ for $ r\ge 0 $. If there is a vertex $v\in 
\partial\Sigma_{R}$, $R>r$, such that
\begin{align*}
    \deg_{-}^{(\Sigma)}(v)+\deg_{0}^{(\Sigma)}(v)\geq 5 ,
\end{align*}
then
$$Z_{R-1}\neq \emptyset.$$
In particular, if $ Z_{R}\neq\emptyset $, then $ Z_{R-1}\neq \emptyset $.
\end{lemma}
\begin{proof}
 The assumption on $ v $ implies  $ \deg_{-}^{(\Sigma)} (v)\ge 3$ or $ \deg_{0}^{(\Sigma)} (v)\ge 3$. 

First assume $ \deg_{0}^{(\Sigma)}(v)\ge 3$.  As $\partial \Sigma_{R}=V(p_{R})  $ by definition and $ p_{R} $ is a simply closed path by Lemma~\ref{l:V_r}, the vertex $ v $ has at most two neighbors in $ \partial\Sigma_{R} $. Thus, there is another neighbor of $ v $ in $ \Sigma_{R} \setminus \partial \Sigma_{R}$. By Lemma~\ref{l:emptyinterior} we infer  $Z_{R-1}\neq \emptyset$.

Now, if $ \deg_{-}^{(\Sigma)} (v)\ge 3$, then the elementary cell $ C_{v} $ which has three vertices in $ \partial \Sigma_{R-1} $ in its boundary and at least one of them (i.e., the ones in the middle) belong to $ Z_{R-1} $.

The ``in particular'' is clear since for $ v\in Z_{R} $  we have by definition $ \deg_{+}^{(\Sigma)}(v)=1 $ and, therefore,   $ \deg_{-}^{(\Sigma)}(v)+\deg_{0}^{(\Sigma)}(v)\geq 5  $. Thus, $ Z_{R-1}\neq \emptyset $.
\end{proof}

The proof of Theorem~\ref{t:triangulation} now follows directly from the following lemma in which we show that emptiness of $ Z_R $ implies the statement of  Theorem~\ref{t:triangulation}.

\begin{lemma}\label{l:Z_Rempty}Let  $G=(V,E)$ be a triangulation such $\deg\ge 6$ outside of $B_{r}$ and $ Z_{r}=\emptyset $ for some $r\ge0$. 
	Then, for all $ R>r $
	\begin{align*}
		\Sigma_R= \partial \Sigma_R=S_{R}\cap U_{r}.
	\end{align*}
	Furthermore, the following statements hold:
	\begin{itemize}
		\item [(a)] For all $ v\in U_{r} $
		\begin{align*}
			\deg_{0}(v)=2  \quad\mbox{and}\quad  1\leq    \deg_{-}(v)\leq 2.
		\end{align*}
		\item [(b)]  $S_R\cap U_{r}$  is given by a  cyclic path  for each $R> r$.
	\end{itemize}	
\end{lemma}

\begin{proof}	
	We first show that  $ \Sigma_R= \partial \Sigma_R$ for all $R>r$.
	Assume by contradiction that  there is $R> r$ such that $ \Sigma_R\setminus \partial \Sigma_R\neq \emptyset$. Then, by  Lemma~\ref{l:emptyinterior}, this implies that $Z_{R-1}\neq \emptyset$. By induction we infer $ Z_{r}\neq\emptyset $ which contradicts the assumption. Thus, $ \Sigma_R= \partial \Sigma_R $.

	We now introduce the two following partitions of $U_r$
	\[
	U_r= \bigcup_{R\geq r+1}  \Sigma_R\qquad\mbox{and}\qquad U_{r}  = \bigcup_{R\geq r+1} S_R \cap U_r.
	\] 
	The first one follows since, by construction, $B_r^{(\Sigma)}= B(p_r)$
	and the second one since $U_r \subseteq V\setminus B_r$. 
	Thus, using $ \Sigma_R = \partial \Sigma_R $ and Lemma~\ref{l:sphere}, 
	we obtain $$ \Sigma_R = \partial \Sigma_R=V_R\subseteq S_R \cap U_r. $$
	Necessarily, as the disjoint union over $ R $ on both sides gives $ U_{r} $, we infer the equality $\Sigma_R = \partial \Sigma_R =V_R= S_R \cap U_r$.
	
	Statement (b)  follows directly since $V_R$ is a  cyclic path  by Lemma~\ref{l:sphere}.

	We now turn to statement (a).
	Let $ v\in U_{r} $. By the considerations above, there  is $ R>r $ such that  $ v\in \Sigma_R$. Since $\Sigma_{R} =\partial \Sigma_R= V(p_R) $ for some  closed path $p_R$, the vertex $ v $ has exactly two neighbors in $ \Sigma_{R} $ and $ \deg^{(\Sigma)}_{0}(v)=2 $.
	So, we infer by Lemma~\ref{l:sphere} 
	$$2=  \deg_{0}^{(\Sigma)}(v)\ge \deg_0(v)\geq 2 . $$
	The inequality $ \deg_{-}(v)\ge 1 $ is obvious. 
	Now, we prove by contradiction that $ \deg_{-}(v)\le 2$.
	Indeed, if $ \deg_{-}(v)\ge 3 $, we have by Lemma~\ref{l:sphere}
	\begin{align*}
		\deg_{0}^{(\Sigma)}(v)+\deg_{-}^{(\Sigma)}(v)=\deg_{0}(v)+\deg_{-}(v)\ge 5.
	\end{align*}
	Lemma~\ref{l:base_case} therefore implies $ Z_{R-1}\neq \emptyset $ and by induction $ Z_{r}\neq \emptyset $. This contradicts the assumption and we infer $ \deg_{-}(v)\le2 $.
\end{proof}

We can now  proceed  to deduce the first geometrical results in the case of triangulations from the two lemmas above.

\begin{proof}[Proof of Theorem~\ref{t:triangulation}~(a)]
 Assume $ \deg\ge 6 $ outside of the root. Obviously, $ Z_{0}=\emptyset $  since $ \deg_{+}(o)=\deg(o) \ge 3$ in a triangulation for the root vertex $ o $. Thus, the ``in particular''  of Lemma~\ref{l:base_case} implies $$  Z_{R}= \emptyset  $$ by induction for all $ R\ge 1 $.
 Hence, the statements of Theorem~\ref{t:triangulation}(a)  follow directly from Lemma~\ref{l:Z_Rempty}~(a) and (b) as well as  the observation that $ U_{0}=V\setminus\{o\} $.	
\end{proof}

\subsubsection{The case $ \deg \ge 7$}
  
 In the case $\deg \geq 7$, we estimate how the size of  $Z_r$ increases exponentially as $r$ decays. 
  
\begin{lemma}[The induction step]\label{l:induction_step} Let $G=(V,E)$ be a triangulation such $\deg\ge 7$ outside of $\Sigma_{r}$, $r\ge0$. Then, for $R>r$
\begin{align*}
    |Z_{R}|\ge 2|Z_{R+1}|.
\end{align*}
\end{lemma}
\begin{proof}
	Assume $ Z_{R+1}\neq \emptyset $. We show that any vertex in $ v\in Z_{R+1} $ induces two vertices in $ Z_{R} $. To this end consider the two distinct neighbors $ w,w'\in \partial \Sigma_{R+1} $ such that $ w\sim v\sim w' $ which exist as $ \partial \Sigma_{R+1}=V(p_{R+1}) $, confer Lemma~\ref{l:V_r}. Say $ w $ is to the left and $ w' $ is to the right of $ v $. We now construct paths from $ v $ in $ C_{v,w} $ and $ C_{v,w'} $ to vertices in $ \partial \Sigma_{R} $. Since we are in a triangulation $ v $ and $ w $ are contained in a triangle in $ C_{v,w} $ which is induced by a common neighbor $ w_{1} $ in $ \Sigma_{R+1}\cup\partial \Sigma_{R} $. If $ w_{1} \in\partial \Sigma_{R}$ we set $ u=w_{1} $ and denote the path $ p=(v,u) $. Otherwise, since $ \deg(w_{1})\ge 7 $, there are at least $ 5 $ neighbors of $ w_{1} $ in $ C_{v,w} $. Starting from $ v $ and counting the vertices along the edges around $ w_{1} $ clockwise,  we pick the third vertex $ w_{2} $ (such that there are two more edges between the edge $ (w_{1},v) $ and $ (w_{1},w_{2}) $). If $ w_{2} $ is not in $ \partial \Sigma_{R} $, then we continue inductively by choosing vertices $ w_{1}, w_{2},\ldots,w_{m} $ in the same manner until we reach $ \partial \Sigma_{R} $. That is having chosen $ w_{j} $ we pick $ w_{j+1} $ to be the third neighbor of $ w_{j} $ starting from $ w_{j-1} $ and counting clockwise.	 We then set $ p=(v,w_{1},\ldots,w_{m}) $ and $ u=w_{m} $.
	 Analogously, we choose a path $ w_{1}',\ldots,w_{n}' $ in $ C_{v,w'} $ with difference of choosing $ w_{j+1}' $ to be the third neighbor of $ w_{j}' $ counter clockwise (instead of clockwise) and set $ w'_{n}=u' $.
	 Finally we pick the path between $ u $ and $ u' $ in $ \partial \Sigma_{R} $ within $ C_{v,w}\cup C_{v,w'} $ and denote it by $ q $. Thus, the paths $ p,p' $ and $ q $ enclose a finite subgraph $ G_{v} $ within  $ C_{v,w}\cup C_{v,w'} $ such that all interior vertices within the paths $ p $ and $ p' $ have degree at least $ 4 $ in $ G_{v} $ by construction. Moreover, also $ v  $ has degree at least $ 4 $ within $ G' $ and any vertex of $ G_{v} $ which is not in the boundary path has degree at least $ 7 $ within $ G_{v} $
	 Thus, by Corollary~\ref{c:CopyPaste} there are at least $ 4 $ vertices in the boundary of $ G_{v} $ with degree $ 3 $ or less in $ G_{v} $. By the consideration about the degrees above, these vertices must be in $ q $. Thus, other than $ u,u' $ there are at least two more vertices $ x,y\in V(q) $ with degree $ 3 $ or less in $ G_{v} $. Since $ x,y\in \partial \Sigma_{R} $ they must have degree $ 3 $ and, therefore, $ x,y\in Z_{R} $.
	 
	 To finish the proof we observe that for two distinct vertices $ v,\tilde v\in Z_{R+1} $ the subgraphs $ G_{v} $ and $ G_{\tilde v}  $ intersect at most in the boundary paths $ p,p' $ and $ \tilde p,\tilde p' $ and, therefore, the vertices $ x,y\in Z_{R} $ and $ \tilde x,\tilde y\in Z_{R} $ are distinct.
	 
	 Thus, any vertex in $ Z_{R+1} $ induces at least two vertices in $ Z_{R} $ and, therefore, $ |Z_{R}|\ge 2|Z_{R+1} |$ follows.
\end{proof}

The proof  Theorems~\ref{t:triangulation} for the case $ \deg \ge 7 $ outside of a finite set uses the same idea as the proof in the case $\deg \geq 6$.

\begin{proof}[Proofs of Theorems~\ref{t:triangulation}(b)]
Assume $ \deg \ge 7 $ outside of $ B_{r} $.
By definition we have $Z_{r}\subseteq \partial \Sigma_{r} \subseteq S_{r}$.
Then, we obtain for  any  $R>r + \log_2|S_{r}|$ by  Lemma~\ref{l:induction_step} and a direct induction
\begin{align*}
|Z_{R}|\leq  \frac{|Z_{r}|}{ 2^{R-r}}<1.
\end{align*}
Thus, $$  Z_{R}= \emptyset, $$
for	$R>r + \log_2|S_{r}|$.
Now, the statements of Theorem~\ref{t:triangulation}(b) follow directly from Lemma~\ref{l:Z_Rempty}(a) and (b).
\end{proof}

\subsection{Other connected components of $V\setminus B_{r}$}\label{s:othercomp}
In this section we show how to extend the statement of Theorem~\ref{t:triangulation}(b) to all $v\in S_{R}  $ with $ R>r+\log_{2}|S_{r}| $ if $ \deg\ge 7 $ outside of $ B_{r} $. Previously, we did this only for vertices in the unbounded connected component $  U_{r} $. Now, we look at the other connected components of $ V\setminus B_{r} $ and discuss how the same arguments as above apply.

Here, we discuss briefly the other finitely many finite connected components. We show that $ B_{R}\cap \mathcal{B}_{r} =\emptyset $ for $ R>r+\log_{2}|S_{r}| $.

Let $ r_0\ge 0 $. Observe that  $ V\setminus B_{r_{0}} $ has finitely many connected  components which are finite. We fix one such connected component and denote it by $ C_{r_{0}}$.
Inductively, we choose a decreasing sequence $ C=(C_{r})_{r\geq r_0} $ 
 $$ C_{r+1}\subseteq C_{r}\subseteq C_{r_{0}} $$ of finite connected components of $ V\setminus B_{r} $, $ r\ge r_{0} $. By finiteness of $ C_{r_{0}} $, we have $ C_{r}=\emptyset  $ for  $ r$ large enough.

For $r \geq r_{0}$, we introduce 
\begin{align*}
V^{(C)}_{r}:=\{v\in B_{r}\mid  \mbox{there is }w\in C_{r}\mbox{ such that }v\sim w  \}. 
\end{align*}
Observe that  $ V_{r}^{(C)}\subseteq C_{r_{0}}$ for $ r>r_{0} $.

Indeed, with the same analysis as in Lemma~\ref{l:V_r} we can show that each $V_r^{(C)}$ is induced by a simple  cyclic path $ p_{r}^{(C)} $ such that
\begin{align*}
V_{r}^{(C)}=V(p_{r}^{(C)}).
\end{align*}
However, a fundamental difference is that the unbounded component $ \mathcal{U}(p^{(C)}_{r}) $ of $ \mathcal{S}\setminus\gamma(p_{r}^{(C)}) $ now includes $ o $ and, moreover,
\begin{align*}
B_{r}\cup U_{r}\subseteq V\cap \overline {\mathcal{U}(p^{(C)}_{r})}=: \overline{U}(p^{(C)}_{r})\quad\mbox{and}\quad C_{r}= V\cap \mathcal{B}(p_{r}^{(C)})=: \mathring{B}(p_{r}^{(C)}) 
\end{align*}
for $ R\geq 1 $.
One can now  define a similar new sphere structure $ \Sigma_{r}^{(C)} $, $ r\ge r_{0}  $, on $ C_{r_{0}} $ by letting $$  \Sigma_{r_{0}}^{(C)}= V\setminus C_{r_{0}}= \overline{U}(p^{(C)}_{r_{0}}),\qquad\partial \Sigma_{r_{0}}^{(C)}=V(p_{r_{0}}^{(C)}) $$
and
$$ \Sigma^{(C)}_{r} =(B_{r}\cap C_{r_{0}})\setminus \Sigma_{r-1}^{(C)},\qquad \partial \Sigma^{(C)}_{r}=V(p^{(C)}_{r}),$$ for $ r> r_{0} $.
As in Lemma~\ref{l:sphere-structure} we can show that,  for $ r,r'\ge r_{0} $ and $ v\in \Sigma_{r}^{(C)} $ and $ w\in \Sigma_{r'}^{(C)} $ with $ v\sim w $, we have $$  |r-r'|\leq 1 .$$
Furthermore, to obtain a similar result as in Lemma~\ref{l:sphere} we denote by $ \deg^{(C)}_{\pm} (v)$ (respectively $ \deg_{0}^{(C)} (v)$) the number of neighbors of $ v\in \Sigma_{r}^{(C)} $ in $ \Sigma_{r\pm 1}^{C}  $
(respectively in $ \Sigma_{r}^{(C)} $). Then, following the arguments given in the proof of Lemma~\ref{l:sphere} we obtain for $ r> r_{0} $
\begin{align*}
\partial \Sigma_{r}^{(C)}=V_{r}^{(C)}
\end{align*}
and on $ \partial\Sigma_{r}^{(C)} $
 $$\deg_{+}\ge \deg_{+}^{(C)} \geq 1,\quad \deg_-=\deg_{-}^{(C)}\ge 1 \quad \mbox{and}\quad2\leq \deg_{0}\leq \deg_{0}^{(C)}$$ 
and on $ \Sigma_r ^{(C)}\setminus \partial \Sigma_r^{(C)}$, 
\[
\deg_+^{(C)} =0.
\]
This does not stand in contradiction to the finiteness of $ C $ since $ \partial \Sigma_{r}^{(C)}=V(p_{r}^{(C)} )=\emptyset$
for large $ r $.

On $ \Sigma_{r}^{(C)} $ we can also define the elementary cells $ C_{v}^{(C)} $ and $ C_{v,w}^{(C)} $ for $ v,w\in \partial\Sigma_{r}^{(C)} $ with $ v\sim w $ as above in Section~\ref{s:elementary-cells}. Finally, one defines the set $ Z_{r}^{(C)} $ of vertices $ v $ in $ \partial\Sigma_{r}^{(C)} $ with $ \deg_{+}^{(C)}=1 $.
In the case, where $ \deg \ge 7 $ outside of $ B_{r_{0}} $ we show as in Lemma~\ref{l:induction_step}
\begin{align*}
|Z_{R}^{(C)}|\ge 2|Z_{R+1}^{(C)}|
\end{align*}
for $ R\ge r_{0} $ and, therefore, 
$$ Z_{R}^{(C)} =\emptyset $$
for $ R>r_{0}+\log_{2}|S_{r_{0}}| $. As in Lemma~\ref{l:base_case}, one sees that $ \deg^{(C)}_{0}(v)+\deg^{(C)}_{-}(v) \ge 5$ for $ v\in \Sigma_{R+1} $  implies $ Z_{R}^{(C)}\neq \emptyset $. So there cannot be vertices of degree larger  $ \deg\ge 7 $ in $ B_{R}\cap C $ for $ R>r_{0}+\log_{2}|S_{r_{0}}| $. In other words $$   B_{R}\cap C_{r_{0}}=\emptyset  $$
for  $ R>r_{0}+\log_{2}|S_{r_{0}}| $.
 
Thus, we have proven $ S_{R}=S_{R}\cap U_{r} $ for  $ R>r_{0}+\log_{2}|S_{r_{0}}| $ and, therefore, we deduce the following generalization of Theorem~\ref{t:triangulation} directly from these theorems.
\begin{thm}\label{t:triangulation_gen} Let $G=(V,E)$ be a planar triangulation.
		If  $\deg\ge 7$ outside of $B_{r_{0}}$ for some $r_{0}\ge0$.  Then,  for all $v \in S_{R}$  such that $R>r_{0}+  \log_{2}|S_{r_{0}}|$,
		\begin{align*}
		\deg_{0}(v)=2  \quad\mbox{and}\quad  1\leq    \deg_{-}(v)\leq 2
		\end{align*}
		and  $S_R$  is a simple cyclic path.
\end{thm}

\begin{rem}\label{rem:deg6}
In order to deduce a conclusion similar to  Theorem~\ref{t:triangulation}(b) in the case $\deg \geq 6$ outside a finite  ball $B_r$, it is sufficient to prove that there is $R\geq r$ such that $Z_{R}=\emptyset$. 
Whether it is possible to prove this remains an open question.
In Section~\ref{s:counter-example} we present an example that satisfies $\deg \geq 6$ outside of  $B_1$ and 
$Z_1= \emptyset$. 
Thus, if we continue this example outside of $B_2$ in any graph $H$, such that it is still a planar triangulation
  with $ \deg \geq 6$ outside of $B_1$, 
  then $H$ will satisfies the conclusion of Theorem~\ref{t:triangulation}(a).
\end{rem}

\subsection{Spanning trees   for triangulations}

In this section, we prove Theorem \ref{t:spanning} on spanning trees for planar triangulations.

Recall the new sphere structure $ \Sigma_{R} $ introduced in Section~\ref{s:sphere-structure}.
Given a vertex $v\in \Sigma_R$ and $w\in V$  such that $v\sim w$, we say that  $ w $ is a \emph{forward (respectively horizontal and backward) neighbor} if 
$w\in \Sigma_{R+1}$ (respectively $\Sigma_R$ and $\Sigma_{R-1}$). 

\begin{thm}\label{t:spanning_triang} Let $G$ be a planar triangulation. Assume one of the following assumptions:
	\begin{itemize}
		\item [(a)]  $ \deg \ge 6 $ outside of the root.
		\item [(b)] 	$\mathrm{deg}\geq 7$ outside of a finite set. 
	\end{itemize}
Then,  then there exists a spanning tree $T$ of $G$ such the vertex degrees of $T$ and $G$ differs at most by $4$  outside of a finite set, where the finite set is empty in case $ \mathrm{(a)} $. Furthermore, $ T $ and $ G $ have the same sphere structure.
\end{thm}

\begin{proof}Let $ r=0 $ in the case (a) and let $r  $ be such that $ \deg \ge 7 $ outside of $ B_{r} $ in the case (b).

We construct the spanning tree inductively and assume we already have chosen a spanning tree of $B_{R}$ in $G$ for $R>\log_2{|S_{r}|}+r$ (which means $ R>0 $ in case (a)).

By Theorem \ref{t:triangulation_gen}, we know that every sphere $ S_{R}\cap U_{r} $ is a simple cyclic path for $ R>r+\log_{2}|S_{r}| $ and by choosing $ R $ even larger we have 
$ S_{R}=S_{R}\cap U_{r} $. (Alternatively one can also apply Theorem~\ref{t:triangulation_gen} so one does not have increase $ R $ further.)

This  implies that every vertex has  two neighbors in the same sphere. For these we remove the connecting edges.
 Since $\mathcal S$ is oriented,  for some fixed $v\in \Sigma_R$,    we can identify the most left and the most right forward neighbor in the next sphere $\Sigma_{R+1}$.  By planarity,  only the most left and the most right forward of these neighbors   $\Sigma_{R+1}$ may have more than one (and thus two) backward neighbors. We consider only  the most right forward neighbor in  $\Sigma_{R+1}$ and  in the case that it has two  backward neighbors, we remove the edge joining $v$.  

For any given vertex we therefore remove at most two horizontal edges, one   edge to a forward neighbor and one  edge to a backward neighbor.
Moreover after this procedure, every vertex has exactly one backward neighbor  and no horizontal edges, meaning that the graph obtained by removing these edges is a tree. 
\end{proof}

\begin{rem}
Using Theorem~\ref{t:triangulation_gen} for which we above sketched the proof allows us to quantify the finite set which is excluded  as the ball $ B_{R} $ with $ R=r+ \log_{2}|S_{r}|$, where $ r $ is such that $ \deg\ge7 $ outside of $ B_{r} $. 
\end{rem}

\subsection{Unique continuation of eigenfunctions for triangulations}\label{s:unique}

In this section, we study the unique continuation of eigenfunctions. Our result is not limited to the Laplacian but holds for general nearest neighbor operators.  We start with a definition.

Let a locally finite graph $G=(V,E)$ be given. A linear operator $A$ defined on a subspace of $C(V)=\C^{V}$ is called a \emph{nearest neighbor operator associated to} $G$, if
if has a matrix representation with respect to the standard basis which is given by some $a:V\times V\to\C$ such that for $v\neq w$,  $$ a(v,w)\neq0\qquad\mbox{ if and only if} \qquad v\sim w .$$
In this case, $A$ acts  as
$$(A\ph)(v)=\sum_{w\in V} a(v,w)\ph(w)=a(v,v)\ph(v)+\sum_{v\sim w} a(v,w)\ph(w)$$
where the sum is finite due to local finiteness of the graph. Moreover, the compactly supported function are included in the domain of definition of $ A $ as $ A $ allows a matrix representation with respect to the standard basis.

The following theorem  is a unique continuation result for eigenfunctions in the case of triangulations. 

\begin{thm}\label{t:EFcompactsupport-triang} Let $G$ be a planar triangulation such that $\mathrm{deg}\geq 7$ outside of a finite set and let $A$ be a  nearest neighbor operator. Then, there are only finitely many linearly independent eigenfunctions of compact support.
\end{thm}
\begin{rem}
By the use of Theorem~\ref{t:triangulation_gen}, one can even quantify the set where the compactly supported  eigenfunctions are supported. Indeed, if $ \deg\ge 7 $ outside of $ B_{r} $, then all compactly supported eigenfunctions are supported in $ B_{R} $ with $ R=r+\log_{2}|S_{r}| $. 
\end{rem}

To prove Theorem~\ref{t:EFcompactsupport-triang} we introduce the polar coordinate representation of nearest neighbor operators as it is used in \cite{FHS,K2}. For a function $\ph\in C(V)$, let $\ph_{r}$ be the restriction of $\ph$ to $C(S_{r})$ and let $s_{r}=|S_{r}|$, $r\ge0$. For a nearest neighbor operator $A$,  let the matrices $E_{r}\in \C^{s_{r+1}\times s_{r}}$, $D_{r}\in\C^{s_{r}\times s_{r}}$, $E_{r}^{+}\in\C^{s_{r}\times s_{r+1}}$ be given such that
$$(A\ph)_{r}=-E_{r-1}\ph_{r-1}+D_{r}\ph_{r}-E_{r}^{+}\ph_{r+1},$$
for all $ \ph $ in the domain of $ A $.
In particular, the matrix $D_{r}$ is  the restriction of $A$ to $C(S_{r})$.

The key point is to prove the matrices $E_r$ are injective. This comes from the following geometric lemma. 

\begin{lemma}\label{l:injective}
Let $G$ be a planar triangulation.
Assume that $\Sigma_R=\partial\Sigma_R$ for all $R> r$ and that $\deg \geq 7$ outside  of $B_{r}^{(\Sigma)}$.
Then for all $R >r$  and all  $v\in \Sigma_R$, there exists $w\in \Sigma_{R+1}$  with $ v\sim w $ and
\[
\deg_{-}(w)=1.
\]
In particular,  for  any nearest neighbor operator associated to $G$ and all $R\ge r$, the matrices $E_R$ are injective.
\end{lemma}
\begin{proof}
Note that the assumption implies that $\deg_{\pm}^{(\Sigma)}=  \deg_{\pm}$ and $ \deg_{0}^{(\Sigma)}=\deg_{0} $ outside of  $B_{r}^{(\Sigma)}$.
Let $v\in   \Sigma_R$ for some $R>r$.
By Theorem~\ref{t:triangulation},
	we have  $ \deg_{-}(v)+\deg_{0}(v)\leq 4 $.    Hence, $ \deg_+^{(\Sigma)}(v)=\deg_{+}(v)\ge 3 $ by the assumption $ \deg\ge7 $ outside of $ B_{R_{0}}^{[\Sigma)} $.
Since all the elementary cells are empty,
this means that $v\in \partial \Sigma_R$ has at least 3 neighbors in $ \Sigma_{R+1}$.	
Recall here that $\Sigma_{R+1}= \partial \Sigma_{R+1}$ is a cyclic closed path.
 Since $G$ is a triangulation, for the vertex $ v\in\partial \Sigma_{R}  $ the forward neighbors induce a path in $\partial \Sigma_{R+1}$ of length $\geq 2$. By planarity, the inner vertices in the path cannot  have another neighbor  in $\partial \Sigma_{R}$ different from $v$.	
	
As a consequence,  the matrix $ E_{R-1} $ is injective. Indeed, let $\ph_{R-1}$ be a non trivial function on $\partial \Sigma_{R-1}$.  That is,  there exists $v \in  \partial \Sigma_{R-1} $ such that $\ph_{R-1} (v)\neq 0$. 
As a consequence, with $ w\in\Sigma_{R} $, $ w\sim v $ and $ \deg_{-}(w)=1 $ as above
\[
E_{R-1}\ph_{R-1}(w)= a(v,w) \ph_{R-1} (v)\neq 0,
\]
and the conclusion follows.
\end{proof}

With the help of the lemma above the proof of Theorem~\ref{t:EFcompactsupport-triang} is along the lines of the proof of \cite[Theorem~9]{K2}.

\begin{proof}[Proof of  Theorem~\ref{t:EFcompactsupport-triang} ]
Let $G$ be a triangulation which satisfies $\mathrm{deg}\geq 7$ outside of a ball and let $R_0\geq r_0 + \log_{2}|S_{r_0}|$. By Lemma~\ref{l:injective} the matrices $E_{R}$ are injective for $R\ge R_{0}$. Let $R\ge R_{0}+1$.
Suppose there is an eigenfunction $\ph$ of $A$ such that  $\ph_{R-1}\not\equiv 0$. Rewriting the eigenvalue equation $(A\ph)_R=\lm\ph_R$ on the $R$-th sphere, one gets
$$E_{R-1}\ph_{R-1}=(D_R-\lm)\ph_R-E_R^{+}\ph_{R+1}.$$
Since $E_{R-1}$ is injective, either $\ph_{R}$ or $\ph_{R+1}$ must be non-zero. Hence, if an eigenfunction is supported on a sphere $S_{R}$ with $R>R_{0}$, then it has infinite support.

Since the space of functions supported on a ball is finite dimensional, there can be at most finitely many linearly independent eigenfunctions of compact support.
\end{proof}

\subsection{Counter-example and continuation to  tessellations} \label{s:counter-example}

%Before we come to the proof of Theorem~\ref{t:EFcompactsupport} we give an example showing, that $\deg\ge6$ or even $\ka_{C}\le0$ outside of a finite set is not enough for the conclusion of Theorem~\ref{t:EFcompactsupport}.
In this section,  we give an example showing, that $\deg\ge6$ or even $\ka_{C}\le0$ outside of a finite set is not enough to exclude compactly supported eigenfunctions such as it was shown in Theorem~\ref{t:EFcompactsupport-triang}.

\begin{figure}[!h]
\scalebox{0.6}{\includegraphics{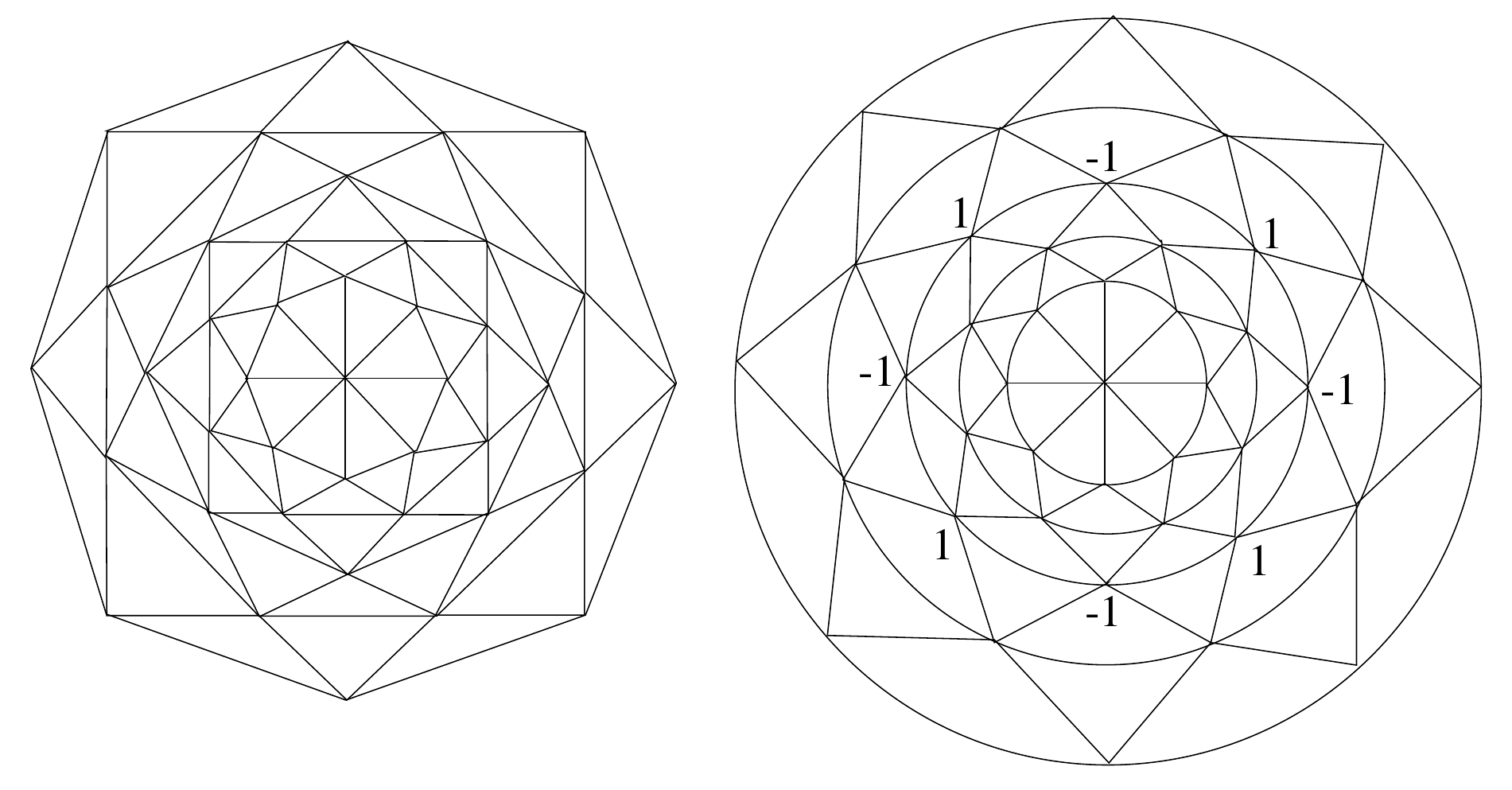}}
\caption{\label{f:example} Two  illustrations of the first 6 distance spheres of  the same tessellation allowing  for infinitely many linearly independent compactly supported eigenfunctions.}
\end{figure}
\begin{eg}\label{e:example1}
The planar graph whose ball $B_{6}$ are pictured in Figure~\ref{f:example} is given such that the root vertex in the middle is adjacent to 8 triangles, the vertices in the first sphere are adjacent to 5 triangles and all further vertices are adjacent to exactly 6 triangles. Hence, $\deg=6$ and $\ka_{C}=0$ outside of $B_{1}$. However, there are  eigenfunctions to the eigenvalue $8$ for $\Delta$ supported on any sphere $S_{r}$, $r\geq 1$. One eigenfunction is illustrated  on the right hand side of Figure~\ref{f:example} in the third sphere.
\end{eg}

In \cite{KLPS} it is shown that $ \ka_{C}\leq 0 $ everywhere implies the absence of compact supported eigenfunction. The example above underscores clearly that our results are not merely simple perturbation results of \cite{KLPS} but the issue in question is much more subtle.

Furthermore, we conclude that graphs as the one above cannot be changed on a finite set such that one obtains a graph with $ \ka_{C}\leq 0 $. 
We finish this section to show that such a procedure  is indeed  possible if one has $ \deg\ge7 $ outside of a finite set.

To this end, we recall the notion of  strictly locally tessellating graphs from \cite{K2} which slightly more general than tessellations as they allow for unbounded faces.
We call a planar locally finite graph $G$   \emph{strictly locally tessellating}  if the following three assumptions are satisfied:
\begin{itemize}
	\item [(T1)] Every edge is included in   two faces.
	\item [(T2)] Every two faces are either disjoint or intersect in one vertex or one edge.
	\item [(T3)] Every face is homeomorphic to a closed disc or to the half space. 
\end{itemize}
In a strictly locally tessellating graph the vertex curvature can seen to be equal to (confer \cite[Lemma~3]{K2})
\begin{align*}
	\ka(v)    =1-\frac{\mathrm{deg}(v)}{2} +\sum_{f\in F, f\ni v}\frac{1}{\mathrm{deg}(f)}.
\end{align*}
In \cite{BP1,BP2} tessellations are considered, that are strictly locally tessellating graphs with the following stronger assumption replacing (T3)
\begin{itemize}
	\item [(T3*)] Every face is homeomorphic to a closed disc.
\end{itemize}

We show that a planar graph with $ \deg\ge7 $ outside of a finite set can be embedded modified on a finite set such that one obtains a strictly locally tessellating graph with non-positive corner curvature. This is uselful as for non-positively corner curved graphs various subtle structural results are known, see e.g.  \cite{BP1,BP2,K2}.

\begin{thm}\label{t:non_positive_curvature} Let $G=(V,E)$ be a planar graph such $\deg\ge 7$ outside of $B_{r}$ for some $r\ge0$. Then, there is a strictly locally tessellating graph $G'=(V',E')$ and $R\ge \log_{2}|S_{r}|+r+1$ such that
	\begin{itemize}
		\item [(a)] $\ka_{C}'\leq 0$.
		\item [(b)] $V'={(}V\setminus B_{R}{)}\cup\{o\}$ and $G'_{V'\setminus \{o\}}=G_{V\setminus B_{R}}$.
		\item [(c)] $d'(\cdot,o)+R=d(\cdot,o)$,
	\end{itemize}
	Moreover, there is a tessellation $G''=(V',E'')$ and $R\ge r$ such that (a), (c) and $E'\subseteq E''$.
\end{thm}

The theorem above directly follows from the next lemma (valid for any planar graphs) as $ \deg'\ge 6 $ implies $ \ka_{C}'\leq 0 $. The underlying idea is to replace a ball around the set with small degree by a single vertex. Since outside of this ball the backward degrees are at most two one changes the overall degree at most by one at these vertices.

\begin{lemma}\label{l:deg} Let $N\geq7$ and a planar graph $G=(V,E)$ be given such that $\mathrm{deg}\geq N$ and $\mathrm{deg}_{-}+\deg_{0}\leq4$ outside of a ball $B_{r}$. Then, there exists $r+1\leq R\leq r+1+\lceil\log_{2}(N-1)\rceil$ and a strictly locally tessellating graph $G'=(V',E')$ such that
	\begin{itemize}
		\item [(a)] $\mathrm{deg}'\geq N-1$.
		\item [(b)] $V'={(}V\setminus B_{R}{)}\cup\{o\}$ and $G'_{V'\setminus \{o\}}=G_{V\setminus B_{R}}$.
		\item [(c)] $d(\cdot,o)=d'(\cdot,o)+R$.
	\end{itemize}
	Moreover, there is a tessellation $G''=(V',E'')$ with  $E'\subseteq E''$ and $R\ge r$ such that $ \mathrm{(a)} $, $ \mathrm{(c)} $ hold.
\end{lemma}
\begin{proof}
	We may assume  $|S_{r+1}|\geq N-1$ and let $R=r$. (Otherwise, we observe that $|S_{r+2}|=2|S_{r+1}|$ since $\mathrm{deg}_{+}=\mathrm{deg}-(\mathrm{deg}_{-} +\mathrm{deg}_{0})\geq N-4\geq 3$. Therefore,  we can continue with $R=r+\lceil\log_{2}(N-1)\rceil$.)
	
	Now, we remove all vertices in $B_{R}(o)$ except for $o$, all edges starting and ending in $B_{R}(o)$ and connect $o$ to all vertices in $S_{R+1}$ by an edge. We denote this graph by $G'=(V',E')$.  Then $V\setminus B_{R}(o)=V'\setminus \{o\}$. Notice that  $d'(o,\cdot)+R=d(o,\cdot)$ on this set.
	The vertex degrees agree on $V\setminus B_{R+1}(o)$ and $V'\setminus B_{1}'(o)$, so, $\mathrm{deg}'\geq N$ on  $V'\setminus B_{R}'(o)$.
	Moreover, $\mathrm{deg}'(v)\ge(\mathrm{deg}(v)-2)+1\geq N-1$ for $v\in S_{1}'(o)$ and $\mathrm{deg}'(o)\geq N$. This shows that $\mathrm{deg}'\geq N-1$ on $G'$. Since $N-1\ge6$, we have $\ka_{C}\le0$. By \cite[Theorem~1]{K2} we conclude that $G'$ is strictly locally tessellating. This shows the first part of the theorem. 
	
	For the ``moreover'' part, we note that a strictly locally tessellating graph can be easily continued to a tessellation by closing unbounded faces by horizontal edges (i.e., edges connecting vertices in the same sphere).
\end{proof}

\section{General planar graphs} \label{s:general-planar}
In this section we show how to carry over the result of Section \ref{s:triangulation}  from triangulations to general planar graphs.
%%%%%%%%%%%%%%%%%%%%%%%%%%%%%%%%%%%%%%%%%%%%%%%%%%%%%%%%%%
\subsection{Triangulation supergraphs}\label{s:proofspanning}

The results of this section are based on next lemma. It says that every planar graphs has a triangulating supergraph with the same sphere structure.

\begin{lemma}\label{l:triangulation} Let $G=(V,E)$ be a locally finite, connected, planar graph and $o\in V$. There is a locally finite, planar triangulation $G'=(V,E')$ with $E\subseteq E'$ and the same sphere structure, i.e.,  $d(v,o)=d'(v,o)$, $v\in V$.
\end{lemma}
\begin{proof} Adding an edge between the vertices $v,w\in V$, $v\neq w$ and $v\not\sim w$ changes the sphere structure of a graph with respect to $o$ if and only if $|d(v,o)-d(w,o)|\ge2$.
Thus, the result can be deduced from the following claim.\\
\emph{Claim:} For every face $f$ in a planar graph which is not a triangle there exist vertices $v_{0},v_{1}\in V\cap f$ with $v_{0}\not\sim v_{1}$ and $|d(v_{0},o)-d(v_{1},o)|\le1$. In the case of an unbounded face there are infinitely many such pairs of distinct vertices.\\
Proof of the claim: 
For a face $f$ which is not a triangle let $v \in V\cap f$ be such that $d(v,o)=\min\{d(w,o)\mid w\in V\cap f\} $. Let $v_{0},v_{1}\in V\cap f$ be  adjacent to $v$. Then, $d(v,o)\leq d(v_{i},o)\leq d(v,o)+1$, $i=0,1$ and, thus, $v_{0}$, $v_{1}$ satisfy the assertion. In the case of an infinite face $f$, there is a two sided infinite sequence where succeeding vertices are adjacent and every vertex of $ f $ is visited. We say that this the boundary walk of $ f $.  By the local finiteness in each sphere there are at most finitely many vertices of this boundary walk. Hence, for a given $n>\min\{d(w,o)\mid w\in V\cap f\}$ there is at least one vertex in $S_{n}(o)$ on each side of the boundary walk and these two are therefore also not adjacent. This proves the claim.
\end{proof}

\subsection{Proofs of the geometric results for general planar graphs}

By  Lem\-ma~\ref{l:triangulation},  a planar graph $G$ with $ \deg \ge 6 $ or $ \deg \ge7 $ has a triangulation supergraph $G'$ that satisfies $ \deg' \ge 6 $ or $ \deg' \ge 7 $.  

We first to turn to the proof of Cartan-Hadamard type result,  Theorem~\ref{t:main}, which we conclude from Theorem~\ref{t:triangulation}.

\begin{proof} [Proof of Theorem \ref{t:main}]
Given a planar graph $ G $, we consider the triangulation supergraph $ G' $ given by Lemma~\ref{l:triangulation} above which has the same vertex set $ V$. Let $ r=0 $ for (a) and $ r $ being the radius such that $ \deg\ge7 $ outside of $ B_{r} $. For this triangulation $ G $, we have  $ 1\leq	\deg_{0}'\leq 2  $ and $  1\leq    \deg_{-}'\leq 2 $ on $ S_{R}\cap U_{r} $ with $ R> r+\log|S_{r}| $. Now, we choose $ R $ even larger such that $ S_{R}\cap U_{r}=S_{R} $ which is possible since $ V\setminus U_{r} $ is a finite set and let $ K=V\setminus B_{R} $.
Since the supergraph $G'$ has the same sphere structure as $G$, we have $\deg_{-}\leq\deg_{-}'$ and $\deg_{0}\leq\deg_{0}'$. Note also that since the graph is connected, $\deg_-\geq 1$ on $V\setminus \{o\}$. Thus, the statement  $ 	\deg_{0}\leq 2  $ and $  1\leq    \deg_{-}\leq 2 $ follows on $ V\setminus K $. This readily gives that geodesics can be continued as $ \deg_{+}=\deg-\deg_{0}-\deg_{-}\ge 3 $ outside of $ K $. Finally, spheres outside of $ K $ in the triangulation supergraph $ G' $  are given by cyclic path which gives that the spheres in $ G $ are cyclically ordered.
\end{proof}

Next we turn to the proof of Theorem \ref{t:spanning} which says that one obtains spanning tree by changing the vertex degree at most by $ 4 $ outside of a finite set.

\begin{proof} [Proof of Theorem \ref{t:spanning}]
The proof follows along the lines for the corresponding proof for  triangulations. Let the finite set $ K $ be chosen as in Theorem~\ref{t:main}. Then spheres  are cyclically ordered outside of $ K $. We first remove the vertices within the spheres which by Theorem~\ref{t:main} changes the vertex degree at most by $ 2 $. Now the cyclic ordering of the spheres allows us to speak of the most right and the most left forward neighbor of a vertex $ v $ in a sphere $ S_{R} $ for large $ R $. By planarity only these two vertices can have more than one backward neighbor.  If the most right forward neighbor has more than one backward neighbor, call it $ w $, we remove this edge. On the other hand, any vertex $ w $ in $ S_{R'} $ with backward degree more than $1$ is a most right forward of some vertex in $ S_{R'-1} $. By Theorem~\ref{t:main}, such a vertex $ w $ satisfies $ \deg_{-}(w)\leq 2 $. By this procedure we remove all cycles in this way without changing the sphere structure of the graph. In summary for each vertex we have removed at most two edges within the same sphere, one edge to a forward and one edge to a backward neighbor which makes at most $ 4 $. This proves the statement.
\end{proof}

We next come to the unique continuation statement for eigenfunctions on general planar graphs. To this end we recall the definition of a nearest neighbor operator from Section~\ref{s:unique}.

\begin{thm}\label{t:EFcompactsupport_general} Let $G$ be a planar graph such that $\mathrm{deg}\geq 7$ outside of a finite set and let $A$ be a  nearest neighbor operator. Then, there are only finitely many linearly independent eigenfunctions of compact support.
\end{thm}
\begin{proof}
Consider the triangulation supergraph $ G' $ of $ G $ given by Lemma~\ref{l:triangulation}.  Now, for a vertex $ v\in S_{R} $ for sufficiently large $ R $ there is a forward neighbor $ w $ such that $ \deg_{-}'(w)=1 $ by Lemma~\ref{l:injective}. Since $ v $ is the only backward neighbor of $ w $ in $ G' $, it must also be a backward neighbor in $ G $ (otherwise $ d(o,w)<d'(o,w) $ contradicting Lemma~\ref{l:triangulation}). Thus, following the argument as in the proof of Theorem~\ref{t:EFcompactsupport-triang} we conclude the statement.
\end{proof}

From this theorem we can immediately deduce that the operator $ \Delta+q $ admits at most finitely many linearly independent eigenfunctions whenever $ \ka_{\infty}=-\infty $.

\begin{proof}[Proof of  Theorem~\ref{t:main2}] Obviously $\Delta+q$ is a nearest neighbor operator and $\ka_{\infty}=-\infty$ implies $\deg\ge7$ outside of a finite set. Hence, the statement follows from Theorem~\ref{t:EFcompactsupport_general}.
\end{proof}
%%%%%%%%%%%%%%%%%%%%%%%%%%%%%%%%%%%%%%%%%%%%%%%%%%%%%%%%%%%%

%%%%%%%%%%%%%%%%%%%%%%%%%%%%%%%%%%%%%%%%%%%%%%%%%%%%%%%%%%%%
\section{Discrete spectrum, eigenvalue asymptotics and decay of eigenfunctions}\label{s:sparseinequality}
In this section we prove Theorem~\ref{c:spareinequality1}, Corollary~\ref{c:spareinequality2} and Theorem~\ref{t:eigenfunctions}. To this end we extend the inequalities presented in \cite{BGK,G} for planar graphs. Here, we use that planar graphs with large vertex degree outside of a finite set
are bounded perturbations of a tree. In particular, an immediate consequence of Theorem~\ref{t:spanning} is the following.

\begin{cor}\label{c:bounded_perturbation} Let $G=(V,E)$ be a planar graph such $\deg\ge 7$ outside of $B_{r}$ for some $r\ge0$. Then, there is a tree $T=(V,E')$ with $E'\subseteq E$ such that $\Delta_{T}$ is a bounded perturbation of $\Delta_{G}$.
\end{cor}

From Corollary~\ref{c:bounded_perturbation}, we derive the following inequality which improves the considerations of \cite{BGK} for planar graphs. These inequalities might be of interest in their own rights.

\begin{thm}\label{t:sparseinequality} Let $G$ be a planar graph such that $ \deg \ge 6 $ outside of the root or $\deg\ge7$ outside of a finite set and $q:V\to[0,\infty)$. Then,
there is $C\ge0$ such that
\begin{itemize}
  \item [(a)] for all $\eps>0$ 
\begin{align*}
    (1-\eps)(\deg+q) -\frac{1}{\eps}-C\leq\Delta+q\leq
        (1+\eps)(\deg+q) +\frac{1}{\eps}+C,
\end{align*}
  \item [(b)] for all $\ph\in C_{c}(V)$, $\|\ph\|=1$, we additionally have
\begin{align*}
    \langle\ph,(\deg+q)\ph\rangle -2\sqrt{\langle\ph,(\deg+q)\ph\rangle}-C
    &\leq
    \langle\ph,(\Delta+q)\ph\rangle\\
    &\leq
    \langle\ph,(\deg+q)\ph\rangle +2\sqrt{\langle\ph,(\deg+q)\ph\rangle}+C.
\end{align*}
\end{itemize}
\end{thm}

\begin{rem}(a) The constant $C$ in the theorem above depends only on the norm of the Laplacian on a  neighborhood of the finite set outside of which we have $\deg\ge7$. In particular, if $\deg\ge7$ everywhere the constant can be chosen $C=0$.\\
%\red{I do not see why $C=0$ since we always delete some edges to compare to the tree.}
(b) The considerations of \cite{BGK} yield an estimate that has $\pm3/\eps$ instead of the constants $\pm(1/\eps+C)$ in (a).
\end{rem}

The essential step in the proof of Theorem~\ref{t:sparseinequality} is to combine Theorem~\ref{t:spanning} with techniques developed in \cite{G,BGK}. Our rather special situation allows for a   very transparent and non-technical treatment.  For sake of being self-contained and  to illustrate the core of the techniques of both \cite{G} and \cite{BGK} we obtain (a) by the Hardy inequality techniques of \cite{G} and (b) by the isoperimetric techniques of \cite{BGK}. Observe that one could also derive (a) from (b) using some technical estimates of \cite{BGK}.

\begin{proof}[Proof of Theorem~\ref{t:sparseinequality}]
By Corollary~\ref{c:bounded_perturbation} there is a tree $T=(V,E')$ such that for the Laplacian $\Delta_{T}$ on the tree there is $C\ge0$ such that $\Delta_{T}-C\leq\Delta\leq\Delta_{T}+C$.
Denote the vertex degree in $T$ by $\deg_{T}$ and observe that by Theorem~\ref{t:spanning} we have $\deg_{T}\leq\deg\leq\deg_{T}+4$.\\
(a) For a positive function $m:V\to(0,\infty)$ let
$q_{m}:V\to \R$ be given by
\begin{align*}
    q_{m}(v)=\deg(v)-\sum_{w\sim v} \frac{m(w)}{m(v)}.
\end{align*}
By direct calculation, which is sometimes refereed to as the ground state representation, (confer \cite[Proposition~1.1]{G} or \cite[Proposition~3.2]{HK}) we obtain for $\ph\in C_{c}(V)$
\begin{align*}
\langle\ph,\Delta_T\ph\rangle&=\frac{1}{2}\sum_{v\sim w} (\ph(v)-\ph(w))^{2}\\
&= \frac{1}{2}\sum_{v\sim w}m(v)m(w) \Big(\frac{\ph(v)}{m(v)}-\frac{\ph(w)}{m(w)}\Big)^{2} +\sum_{v\in V} q_{m}(v)\ph(v)^{2}
\end{align*}
and, therefore, $ \Delta_{T}\ge q_{m}$ on $C_{c}(V)$. Now, for $\eps>0$,  we choose $m(v)=\eps^{d(v,o)}$ and observe  $q_{m}=(1-\eps)\deg_{T}-1/\eps$. Thus,
\begin{align*}
    (1-\eps)\deg_{T}-\frac{1}{\eps}\leq\Delta_{T}
\end{align*}
on $C_{c}(V)$. Since $T$ is a tree, the operator $\Delta_{T}$ is unitarily equivalent to the operator $2\deg-\Delta_{T}$ (the unitary operator is multiplication by $(-1)^{d(\cdot,o)}$). Hence, we conclude
\begin{align*}
    \Delta_{T}\leq 2\deg_{T}-\Delta_{T} \leq (1+\eps)\deg_{T}+\frac{1}{\eps}
\end{align*}
on $C_{c}(V)$. Statement (a) follows now from   $\Delta_{T}-C\leq  \Delta\leq
\Delta_{T} +C$ and $\deg_{T}\leq\deg\leq\deg_{T}+4$ discussed in the beginning of the proof and the assumption $q\ge0$.\smallskip

(b) By Theorem~\ref{t:non_positive_curvature} the tree $T$ has degree larger than $2$ outside of a finite set $K$. Since $T$ is a spanning tree and, hence, connected it has degree greater or equal to $1$ everywhere. We define
\begin{align*}
    d_{T}=\deg_{T}+q' \quad\mbox{with}\quad q'=q+1_{K}.
\end{align*}
By the discussion above $d_{T}\ge \deg_{T}+1_{K}\ge 2$.
We further notice that on a tree any subgraph $T_{W}=(W,E'_{W})$ of $T=(V,E')$ induced by a finite set $W\subseteq V$ satisfies $|E'_{W}|\leq |W| $
(confer \cite[Lemma~6.2]{BGK}). This implies $$d_{T}(1_{W})=2|E'_{W}|+|\partial W|+q'(1_{W})\leq2|{W}|+|\partial W|+q'(1_{W}),$$
where $\partial W=\{(v,w)\in W\times V\setminus W\mid v\sim w\}$ and $q'(\ph)=\sum_{v\in V}\ph(v)^{2}q'(v)$, $\ph\in C_{c}(V)$.
Let $\ph\in C_{c}(V)$, $\|\ph\|=1$.
Using  an area and a co-area formula (cf.\ \cite[Theorem~12 and Theorem~13]{KL2}) with $\Om_{t}:=\{v\in V\mid |\ph(v)|^{2}>t\}$,
and   the discussion above, we obtain
\begin{align*}
 \langle& \ph,(d_{T}-2)\ph\rangle
=\int_{0}^{\infty}\Big(d_{T}(1_{\Om_{t}})-2|\Om_{t}|\Big)dt
\leq \int_{0}^{\infty}|\partial\Om_{t}|+q'(1_{\Om_{t}})dt \\
&=\frac{1}{2}\sum_{v\sim w} \left|\ph(v)^{2}-\ph(w)^{2}\right| + q'(\ph)=\frac{1}{2}\sum_{v\sim w} \left|(\ph(v)-\ph(w))(\ph(v)+\ph(w))\right| + q'(\ph)
\\
&\leq \frac{1}{2} \left(\sum_{v\sim
    w}|\ph(v)-\ph(w)|^2+ 2q'(\ph)\right)^{1/2}\left(\sum_{v\sim
    w}|\ph(v)+\ph(w)|^2+ 2q'(\ph)\right)^{1/2}\\
&=   \langle \ph,(\Delta_T+q')\ph\rangle^{\frac{1}{2}} \big(2\langle\ph,d_{T}\ph\rangle- \langle\ph,(\Delta_T+q')\ph\rangle\big)^{\frac{1}{2}}.
\end{align*}
Since $d_{T}\ge2$, we have $\langle \ph, (d_{T}-2)\ph\rangle\ge 0$ and, thus, we can square both sides of the inequality to obtain after reordering the terms,
\begin{align*}
\langle\ph,(\Delta_T+q')\ph\rangle^{2}-2\langle \ph,d_{T}
\ph\rangle \langle \ph,(\Delta_T+q')\ph\rangle+\langle \ph,(d_{T}-2)\ph\rangle^2\le0.
\end{align*}
Resolving the inequality and using $\|\ph\|=1$ yields
\begin{align*}
    \langle \ph,d_{T}\ph\rangle -2\sqrt{\langle \ph, d_{T}\ph\rangle-1}\leq
       \langle\ph,(\Delta_T+q')\ph\rangle\leq
          \langle d_{T}\ph,\ph\rangle +2\sqrt{\langle \ph,d_{T}\ph\rangle-1}.
\end{align*}
Now, we further observe that
\begin{align*}
   \sqrt{\langle \ph, d_{T}\ph\rangle-1}=    \sqrt{\langle \ph, (\deg_{T}+q+1_{K})\ph\rangle-1}\leq\sqrt{\langle \ph, (\deg_{T}+q)\ph\rangle}.
\end{align*}
Thus, (b) follows from the inequalities subtracting $\langle \ph, 1_{K}\ph\rangle$ from the inequalities and using $\Delta_{T}-C\leq\Delta\leq\Delta_{T}+C$ and $d_{T}\leq\deg+q\leq d_{T}+C$.
\end{proof}

Next, we turn to the proof of Theorem~\ref{c:spareinequality1}.

\begin{proof}[Proof of Theorem~\ref{c:spareinequality1}]
For a potential $q\in K_{\al}$, $\al\in(0,1)$, there is $C_{\al}$ such that $q_{-}\leq\al(\Delta+q_{+})+C_{\al}$. We deduce from Theorem~\ref{t:sparseinequality}~(a)  (confer \cite[Lemma~A.3]{BGK})
\begin{align*}
\frac{(1-\al)(1-\eps)}{(1-\al(1-\eps))}(\deg+q) -\frac{(1-\al)(1/\eps+C)+\eps C_{\al}}{(1-\al(1-\eps))}\leq \Delta+q, \quad \mbox{ on $C_{c}(X)$}
\end{align*}
 for all $\eps>0$.\\
By an application of the  Min-Max-Principle, Theorem~\ref{t:minmax},
the  spectrum of $\Delta+q$ is purely discrete if the spectrum of $\deg+q$ is purely discrete. On the other hand, if there are vertices $v_{n}$  such that $(\deg+q)(v_{n})\leq C$ for some $ C $, then $\langle \Delta 1_{\{v_{n}\}},1_{\{v_{n}\}}\rangle=(\deg+q)(v_{n})\leq C$, $n\ge0$. By a Persson-type theorem, \cite[Proposition~2.1]{HKW} we conclude that the bottom of the essential spectrum of $\Delta+q$ is bounded from above by $C$. Hence, the essential spectrum of $\Delta+q$ is non-empty. We summarize that the spectrum of  $\Delta+q$
is purely discrete if and only if $\sup_{K\subset V\,\mathrm{finite}}\inf_{v\in V}(\deg(v)+q(v))=\infty$. Since
\begin{align*}
    -\frac{\deg(v)}{2}\leq\ka(v)\leq 1-\frac{\deg(v)}{6}
\end{align*}
(due to $\deg(f)\ge 3$), this in turn is equivalent to
 $$ \sup_{K\subset V\,\mathrm{finite}}\inf_{v\in V}(-\ka(v)+q(v))=\infty.$$
Next, we assume $q\ge0$. The eigenvalue asymptotics follow directly from  Theorem~\ref{t:sparseinequality}~(b) and  the  Min-Max-Principle, Theorem~\ref{t:minmax} as $x\mapsto x-2\sqrt{x}$ is continuous and monotone increasing on $[1,\infty)$ and $\lm_{0}(\deg+q)\ge1$.
\end{proof}

Now, we turn to the proof of Corollary~\ref{c:spareinequality2}.
\begin{proof}[Proof of Corollary~\ref{c:spareinequality2}]
If the face degree is constantly $k$  outside of a compact set $K\subseteq X$, then
\begin{align*}
    \ka(v)=1-\frac{k-2}{2k}\deg(v),
\end{align*}
for $v\in V\setminus K$.  The eigenvalue asymptotics follow now  from  Theorem~\ref{c:spareinequality1}.
\end{proof}

%%%%%%%%%%%%%%%%%%%%%%%%%%%%%%%%%%
{Finally, we prove Theorem~\ref{t:eigenfunctions} on the decay of eigenfunctions.} The proof we give here is similar to the techniques developed of \cite{KePo}. However,  for the convenience of the reader we include a short proof. Indeed, in our situation of planar graphs with large degree, the proof simplifies even substantially.
\begin{proof}[Proof of Theorem~\ref{t:eigenfunctions}]
%We denote for functions $ \ph\in C_{c}(V) $	
%	 \[	|\nabla \ph |^2(x):= \frac{1 }{2} \sum_{y, y\sim x} (\ph(x)- \ph(y))^2.	\]
A direct computation, often referred to as the ground state representation as used in the proof of Theorem~\ref{t:sparseinequality} above (confer \cite[Proposition~1.1]{G} or \cite[Proposition~3.2]{HK}), gives for  {$\ph$ with compact support}, 
$u \in  D(\Delta) $ and the basic estimate $ 2ab
\leq a^{2}+b^{2} $
\begin{align*}
\frac{1}{2}\sum_{x,y, y\sim x} u(x)^{2}(\ph(x)- \ph(y))^2&\ge\frac{1}{2} \sum_{x,y,  x\sim y} u(x) u(y) ( \ph(x) - \ph(y)) ^2\\
	&= \langle \ph u, \Delta( \ph u)\rangle-\langle \ph^2 u, \Delta  u\rangle.
\end{align*}
If  $u\in  D(\Delta) $ is an eigenvector of $\Delta$ with eigenvalue $\lambda$, we estimate together with  the form bound  from Theorem~\ref{t:sparseinequality}~(a) 
\begin{align*}
\frac{1}{2}\sum_{x,y, y\sim x} u(x)^{2}(\ph(x)- \ph(y))^2 \ge  \langle((1-{\eps})\deg-C_{\eps}-\lambda ) \ph u, \ph u \rangle
 \end{align*}
 {for  $0<\eps<1$ and }with $C_{\eps} =1/\eps +C$  some $ C\ge0 $.
We define for $ N \ge 0$
\begin{align*}
	\ph_{N}= 1_{B_{N}}\al^{d(\cdot,o)} + 1_{B_{2N}\setminus B_{N}}\al^{2N-d(\cdot,o)}
\end{align*}
with 	$ \alpha= \sqrt{2}(1-\ve) {+} 1$
and observe that  {for $\eps >0$ small enough},
\begin{align*}
	\frac{1}{2}\sum_{y, y\sim \cdot}(\ph_{N}(\cdot)- \ph_{N}(y))^2\leq (1-\eps)^{2} \ph^{2}_{N}\deg
	+{\frac{1}{2}} 1_{ S_{2N+1}}\deg_{-}
\end{align*}
where we used that $ \ph_{N}(x)-\ph_{N}(y)=1 $ for $ x\in S_{2N} $ and $ y\in S_{2N+1} $.
Combining this with the estimate above, we obtain after reordering the terms
\begin{align*}
	\frac{1}{2}\sum_{x\in {S_{2N+1}}}\deg_{-}(x)u(x)^{2}\ge\langle (\eps(1-{\eps})\deg-C_{\eps}-\lambda){\ph_N u, \ph_N u \rangle}.
\end{align*}
Since we assumed that $ \deg $ becomes arbitrarily large outside of finite sets, there is a finite set $ K $ {and a constant $c_\ve >0$} such that $ \eps(1-\eps)\deg-C_\eps-\lambda \ge  {c_\ve \deg } $ outside of $ K $. Furthermore, by Theorem~\ref{t:main}, $\deg_{-}  $ is a bounded function.  Thus,
there is $ C'=C'_{\eps} $ such that for all $ N $
	\begin{align*}
		C'\|u\|^{2}\ge \sum_{x\in X}\deg(x)\ph_{N}^{2}(x)u^{2}(x)\ge \sum_{x\in B_{N}}\deg(x)\al ^{2d(o,x)}u^{2}(x).
	\end{align*}
By monotone convergence, we conclude
\begin{align*}
		{C'}\|u\|^{2}\ge \sum_{x\in X}\deg(x)\al ^{2d(o,x)}u^{2}(x).
\end{align*}
 This finishes the proof.
\end{proof}

\appendix
\section{Applications of the Min-Max-Principle}
In this appendix we shortly discuss an application of the Min-Max-Principle to (non-linear) functions of operators.

Let $H$ be a Hilbert space with norm $\|\cdot\|$. For a quadratic form $Q$, denote the form norm by $\|\cdot\|_{Q}:=\sqrt{Q(\cdot)+\|\cdot\|^{2}}$. For a
selfadjoint operator $A$ which is bounded from below, we denote the bottom of the essential
spectrum by $\lm_{0}^{\mathrm{ess}}(A)$. Let $n(A)\in \N_{0}\cup\{\infty\}$ be the dimension of the range of the spectral projection of the interval $(-\infty,\lm_{0}^{\mathrm{ess}}(A))$.  Whenever $\lm_{0}(A)<\lm_{0}^{\mathrm{ess}}(A)$, we denote the eigenvalues below $\lm_{0}^{\mathrm{ess}}(A)$  by $\lm_{n}(A)$, for $0  \leq n \leq n(A)$,  in increasing  order counted with multiplicity.

\begin{thm}\label{t:minmax} Let $(Q_{1},D(Q_{1}))$ and $(Q_{2},D(Q_{2}))$ be closed symmetric non-negative quadratic forms with a common form core $D_{0}$  and let the corresponding selfadjoint operators be denoted by $A_{1}$ and $A_{2}$.
Assume there are continuous monotone increasing functions $f_{1},f_{2}:[\lm_{0}(A_{2}),\infty)\to\R$ such  for all $\ph\in D_{0}$ with $\|\ph\|=1$
\begin{align*}
    f_{1}(Q_{2}(\ph))\leq Q_{1}\leq    f_{2}(Q_{2}(\ph)).
\end{align*}
Then, for $0 \leq n \leq \min(n(A_1), n(A_2))$,
$$f_{1}(\lm_{n}(A_{2}))\leq \lm_{n}(A_{1})\leq
    f_{2}(\lm_{n}(A_{2})).$$
Moreover, if $\lim_{r\to\infty}f_{1}(r)=\lim_{r\to\infty}f_{2}(r)=\infty$, then  $\si_{\mathrm{ess}}(A_{1})=\emptyset$ if and only if $\si_{\mathrm{ess}}(A_{2})=\emptyset$.
\end{thm}
\begin{proof}
Letting
\begin{align*}
    \mu_{n}(A):=\sup_{\ph_{1},\ldots,\ph_{n}\in H} \inf_{\psi\in\{\ph_{1},\ldots,\ph_{n}\}^{\perp}\cap D_0,\ \|\psi\|=1}{Q(\psi)},
\end{align*}
for a selfadjoint operator $A$ with form $Q$ and $D_{0}\subseteq D(Q)$, we have by the Min-Max-Principle \cite[Chapter XIII.1]{RS} $\mu_{n}(A)=\lm_{n}(A)$ if $\mu_{n}(A)<\lm_{0}^{\mathrm{ess}}(A)$
and $\mu_{n}(A)=\lm_{0}^{\mathrm{ess}}(A)$ otherwise, $n\ge0$.
Now, observe that for a continuous monotone increasing function $f:[0,\infty)\to\R$ and a function $g:X\to[0,\infty)$ defined on an arbitrary set $X$  we have
\begin{align*}
    \inf_{x\in X}f(g(x))=f\left(\inf_{x\in X} g(x)\right).
\end{align*}
Now, assume $n\leq\min\{n(A_{1}),n(A_{2})\}$ and  let $\ph_{0}^{(j)},\ldots,\ph_{n}^{(j)}$ be the eigenfunctions of $A_{j}$ to $\lm_{0}(A_{j}),\ldots,\lm_{n}(A_{j})$ and denote $$ U_{j}^{(n)}:=\{\ph_{1}^{(j)},\ldots,\ph_{n}^{(j)}\}^{\perp}\cap \{\psi\in D_0\mid \|\psi\|=1\},\qquad j=1,2 .$$
We apply the discussion above with $f=f_{1}$ and $$ g=g_{1}:U_{2}^{(n)}\to[0,\infty),\qquad \psi\mapsto {Q_{2}(\psi)} ,$$ first and  $f=f_{2}$ and $$ g=g_{2}:U_{1}^{(n)}\to[0,\infty),\qquad\psi\mapsto {Q_{2}(\psi)}, $$   later on, to obtain
\begin{align*}
f_{1}(\lm_{n}(A_{2})) &=f_{1}\left(\inf_{\psi\in U^{(n)}_{2}}{Q_{2}(\psi)}\right) =\inf_{\psi\in U^{(n)}_{2}} f_{1}\Big({Q_{2}(\psi)}\Big)
\leq\inf_{\psi\in U^{(n)}_{2}} {Q_{1}(\psi)}\leq\mu_{n}(A_{1})\\
&=\lm_{n}(A_{1})
=\inf_{\psi\in U^{(n)}_{2}} {Q_{1}(\psi)}
\leq \inf_{\psi\in U^{(n)}_{2}} f_{2}\Big({Q_{2}(\psi)}\Big)= f_{2}\left(\inf_{\psi\in U^{(n)}_{2}} {Q_{2}(\psi)}\right) \\
&\leq f_{2}(\mu_{n}(A_{2}))= f_{2}(\lm_{n}(A_{2})).
\end{align*}
This
directly implies the first statement. Assuming now  $\lm_{0}^{\mathrm{ess}}(A_{2})=\infty$ implies $n(A_{2})=\infty$ and $\lim_{n\to\infty}\lm_{n}(A_{2})=\infty$ and, therefore,
 $\lim_{n\to\infty}f_{1}(\lm_{n}(A_{2}))=\infty$, by the assumption on $f_{1}$. Hence, by the above we get $\lm_{0}^{\mathrm{ess}}(A_{1})=\infty$. The other implication follows analogously.
\end{proof}

\textbf{Acknowledgement}.
MK enjoyed the hospitality of Bordeaux University when this work started. Moreover, MK is grateful to Daniel Lenz for inspiring discussions on the subject and acknowledges the financial support of the German Science Foundation (DFG), Golda Meir Fellowship, the Israel Science Foundation (grant no. 1105/10 and  no. 225/10)  and  BSF grant no. 2010214.

\end{document}